\documentclass[10pt]{amsart}
\usepackage{amsmath}
\usepackage{amsfonts}
\usepackage{amssymb}
\usepackage{amsthm}
\usepackage{url}
\usepackage[all]{xy}
\usepackage{dsfont} 
\usepackage{graphicx}
\usepackage{caption}
\usepackage{subcaption}
\usepackage{comment} 
\usepackage{stmaryrd}
\usepackage{hyperref}
\usepackage{todonotes}
\usepackage{color}
\usepackage{enumerate,tikz-cd}
\usepackage{fixltx2e}
\usepackage{MnSymbol}
\usepackage{mathtools}
\allowdisplaybreaks

\numberwithin{equation}{section}

\newtheorem{proposition}{Proposition}[section]
\newtheorem{lemma}[proposition]{Lemma}
\newtheorem{theorem}[proposition]{Theorem}
\newtheorem{corollary}[proposition]{Corollary}

\theoremstyle{definition}
\newtheorem{remark}[proposition]{Remark}
\newtheorem{definition}[proposition]{Definition}
\newtheorem{example}[proposition]{Example}

\DeclareMathOperator{\Bl}{Bl}

\DeclareMathOperator{\Aut}{Aut}

\DeclareMathOperator{\NA}{NA}

\DeclareMathOperator{\ord}{ord}

\DeclareMathOperator{\Vol}{Vol}
\DeclareMathOperator{\val}{val}
\DeclareMathOperator{\di}{div}
\DeclareMathOperator{\red}{red}
\DeclareMathOperator{\Amp}{Amp}
\DeclareMathOperator{\dmu}{d\mu}
\DeclareMathOperator{\Ent}{Ent}

\newcommand{\R}{\mathbb{R}}
\newcommand{\C}{\mathbb{C}}
\newcommand{\A}{\mathbb{A}}

\newcommand{\Q}{\mathbb{Q}}

\newcommand{\pr}{\mathbb{P}}
\renewcommand{\epsilon}{\varepsilon}

\newcommand{\M}{\mathcal{M}}

\newcommand{\scO}{\mathcal{O}}

\newcommand{\mfa}{\mathfrak{a}}

\newcommand{\B}{\mathcal{B}}

\newcommand{\E}{\mathcal{E}}
\renewcommand{\L}{\mathcal{L}}
\newcommand{\X}{\mathcal{X}}
\newcommand{\Y}{\mathcal{Y}}
\renewcommand{\H}{\mathcal{H}}

\renewcommand{\phi}{\varphi}

\newcommand\MA{\mathrm{MA}}

\newcommand\an{^{\mathrm{an}}}

\newcommand\triv{{\mathrm{triv}}}

\title{On divisorial stability of finite covers}
\author{Ruadha{\'i} Dervan}
\address{School of Mathematics and Statistics, University of Glasgow, University Place, Glasgow, G12 8QQ, United Kingdom}
\email{ruadhai.dervan@glasgow.ac.uk}

\author{Theodoros Stylianos Papazachariou}
\address{Isaac Newton Institute, University of Cambridge, 20 Clarkson Road, Cambridge, CB3 0EH, United Kingdom}
\email{tsp35@cam.ac.uk}
\begin{document}
	\maketitle
	
	\begin{abstract}
		Divisorial stability of a polarised variety is a stronger---but conjecturally equivalent---variant of uniform K-stability introduced by Boucksom--Jonsson. Whereas uniform K-stability is defined in terms of test configurations, divisorial stability is defined in terms of convex combinations of divisorial valuations on the variety.
		
		We consider the behaviour of divisorial stability under finite group actions, and prove that equivariant divisorial stability of a polarised variety is equivalent to log divisorial stability of its quotient. We use this and an interpolation technique to give a general construction of equivariantly divisorially stable polarised varieties.
		
	\end{abstract}
	
	\section{Introduction}\label{intro}
	
	The theory of K-stability of Fano varieties has achieved its prominence due to its links both with K\"ahler geometry (through the existence of K\"ahler--Einstein metrics \cite{CDS,RB,GT}) and moduli theory (through the construction of  moduli spaces of K-polystable Fano varieties, see \cite{xu-survey, Liu_Xu_Zhuang_2022}). There are essentially two reasons why the algebro-geometric theory of K-stability of Fano varieties has been so successful: the first is the interplay with birational geometry and the minimal model programme (originating in \cite{YO, lixu}), while the second is the reinterpretation of K-stability in terms of divisorial valuations on the Fano variety \cite{fujita, li}.
	
	K-stability is also of interest for general polarised varieties (projective varieties endowed with an ample line bundle), and in this situation there is still a substantial literature  linking K-stability with K\"ahler geometry through the existence of constant scalar curvature K\"ahler metrics (namely the \emph{Yau--Tian--Donaldson conjecture} \cite{STY,  GT,SD-toric}). However, the algebro-geometric theory of K-stability of general polarised varieties is considerably less developed than its Fano counterpart and relatively little is known. While one cannot expect birational geometry to play as significant a role in this generality, it is still reasonable to attempt to use valuative tools in studying K-stability of arbitrary polarised varieties. With this in mind, the first author and Legendre introduced a notion of \emph{valuative stability} of a polarised variety \cite{dervan-legendre}, which should be \emph{strictly weaker} than K-stability for general polarised varieties, although it is equivalent in the Fano situation.
	
	The more powerful notion of \emph{divisorial stability}, very recently introduced by Boucksom--Jonsson \cite{bj_non_arch_ii}, associates numerical invariants to  \emph{convex combinations} of divisorial valuations. By their work, divisorial stability implies---and is conjecturally equivalent to---\emph{uniform} K-stability, which in turn is conjecturally equivalent to the existence of constant scalar curvature K\"ahler metrics when the variety is smooth. In fact, the same conjecture that would lead to a resolution of the ``uniform version'' of the Yau--Tian--Donaldson conjecture (through \cite{BBJ, chi-li-csck}) would also imply that divisorial stability is equivalent to uniform K-stability \cite{bj_non_arch_ii}. There is already some evidence that divisorial stability is a more useful notion than uniform K-stability, through Boucksom--Jonsson's proof that divisorial stability is an \emph{open} condition in the ample cone \cite[Theorem A]{bj_non_arch_ii} (see Liu for prior work in the setting of valuative stability \cite{yaxiong-liu}).
	
	Thus, it is hoped that divisorial stability will produce a richer theory of stability of polarised varieties, by analogy with the Fano situation. The goal of this paper is to showcase another situation in which divisorial stability appears more useful than the traditional approach. We denote by $(X,L_X)$ and $(Y,L_Y)$ normal polarised varieties, such that $\pi: (Y,L_Y) \to (X,L_X)$ is a Galois cover with Galois group $G$, by which we mean that $G$ acts on $(Y,L_Y)$ in such a way that its quotient by $G$ is $(X,L_X)$. In addition let $\Delta_X$ and $\Delta_Y$ be effective $\Q$-divisors such that $$K_Y+\Delta_Y = \pi^*(K_X+\Delta_X),$$ where we assume both sides are $\Q$-Cartier divisors.

	\begin{theorem}\label{intromainthm}
$((Y,\Delta_Y),L_Y)$ is $G$-equivariantly divisorially stable if and only if $((X,\Delta_X);L_X)$ is log divisorially stable. 
	\end{theorem}
	
	The most applicable special case is when $G$ is cyclic of degree $m$,  $\Delta_Y$ is taken to be trivial and $\Delta_X$ is the integral divisor such that Riemann--Hurwitz produces $K_Y=\pi^*(K_X+(1-1/m)\Delta_X)$. Theorem \ref{intromainthm} then gives the following corollary:
	
	\begin{corollary} 
	$(Y,L_Y)$ is $G$-equivariantly divisorially stable if and only if $((X,(1-1/m)\Delta_X);L_X)$ is log divisorially stable. \end{corollary}

	Analogous results holds for divisorial semistability. The proof compares the divisorial measures used to define divisorial stability of $(Y,L_Y)$ to the corresponding objects on $(X,L_X)$, and uses non-Archimedean geometry to compare various associated numerical invariants. The advantage of divisorial stability over K-stability is analogous to the advantage exploited by Boucksom--Jonsson in their work on openness of divisorial stability in the ample cone: the numerical invariants involved in the definition of divisorial stability involve an entropy  (or log discrepancy) term that is \emph{easier} to manage than the analogous quantity involved in K-stability, while the \emph{energy} (or \emph{norm}) terms become more complicated; much of the proof involves understanding the behaviour of these energy terms under finite covers. We emphasise again this key advantage of divisorial stability: while handling the energy terms becomes more involved than with the traditional approach, one should expect these terms to generally be more manageable (as is the case both in the present work and in Boucksom--Jonsson \cite{bj_non_arch_ii}). By contrast, the entropy term---behind much of the difficulty of K-stability---becomes considerably simpler to understand.
	
	The corresponding result in the \emph{Fano} case was proven in steps by several authors; the first author proved one direction \cite{dervan_2016} for cyclic groups $G$, with the other direction and various improvements being proven by Fujita \cite[Corollary 1.7]{fujita-plt} and Liu--Zhu \cite{liu_zhu_2021}. This result has found many applications to the construction of new examples of K-stable Fano varieties (beyond these three original papers, we give \cite{liu-double} as a typical application), and in the study of K-moduli of log Fano pairs and moduli spaces of K3 surfaces (through \cite{ascher_devleming_liu_2021}). While the technique of \cite{dervan_2016} uses the language of K-stability, the techniques of \cite{fujita-plt,liu_zhu_2021} instead use divisorial valuations. It seems challenging, however, to adapt the techniques of \cite{dervan_2016} to prove an analogous result for general polarised varieties, as the proof given there relies on properties of K-stability of Fano varieties that do not hold more generally \cite{lixu}. We thus emphasise that divisorial stability seems more suited to this problem, as Theorem \ref{intromainthm} \emph{exactly} generalises the Fano results to general polarised varieties. We also mention that the techniques we employ to prove Theorem \ref{intromainthm} are quite distinct from those of \cite{fujita-plt,liu_zhu_2021}, since divisorial stability involves convex combinations of divisorial valuations, and the actual numerical invariants have a somewhat different flavour in the Fano situation.
	
	We use an interpolation technique to produce examples.
	
\begin{theorem}\label{introinterpolation}	
Let $(X,L_X)$ be a divisorially semistable normal polarised variety. There is a $k>0$ such that if we let
\begin{enumerate}[(i)]
\item  $\Delta_X \in |kL_X|$ be such that $(X,\Delta_X)$ is log canonical,
\item and let $\pi: Y \to X$ be the $m$-fold cover of $X$ branched over $\Delta_X$,
\end{enumerate}then $(Y,L_Y)$ is $G$-equivariantly divisorially stable, where $G$ is the associated cyclic group of degree $m$ and $L_Y = \pi^*L_X$.
\end{theorem}

The construction applies for any $m>0$. The $k$ needed depends explicitly on the geometry of $(X,L_X)$, see Remark \ref{whichk}. The proof shows that $((X,\Delta_X);L_X)$ is automatically divisorially stable, meaning by interpolation so is $((X,(1-1/m)\Delta_X);L_X)$, hence by Theorem \ref{intromainthm} $(Y,L_Y)$ is $G$-equivariantly divisorially stable. Although the hypotheses themselves are different, this result is analogous to \cite[Corollary 1.2]{dervan_2016}, where an interpolation strategy was used to give a sufficient condition for K-stability of finite covers of Fano varieties. This was the source of many of the examples of K-stable Fano varieties produced by the K-stability analogue of Theorem \ref{intromainthm} in the Fano setting.

When the output $(Y,L_Y)$ is smooth, and the field we work over is $\C$, this is sufficient to produce constant scalar curvature K\"ahler metrics.

\begin{corollary}\label{introcor}
Under the same hypothesis as Theorem \ref{introinterpolation}, provided $(Y,L_Y)$ is smooth, $c_1(L_Y)$ admits a constant scalar curvature K\"ahler metric.
\end{corollary}

This corollary relies on an equivariant version of the result of Boucksom--Jonsson relating divisorial stability to uniform K-stability on $\E^1$, and work of Li producing constant scalar curvature K\"ahler metrics from $G$-equivariant uniform K-stability on $\E^1$. 

There are many analytic counterparts to the results mentioned above, all under smoothness assumptions. The usage of finite symmetry groups in the study of K\"ahler--Einstein metrics goes back at least to Siu \cite{siu}, Nadel \cite{nadel} and Tian \cite{tian}, and general results more in the spirit of our work were proven by Arezzo--Ghigi--Pirola \cite{AGP} and Li--Sun \cite{li-sun}. In the general constant scalar curvature setting,  Aoi--Hashimoto--Zheng have proven one part of the analogue of Theorem \ref{introinterpolation} \cite[Theorem 1.10]{AHZ}, namely the existence of constant scalar curvature K\"ahler metrics with cone angle singularities along $\Delta_X$ for sufficiently large $k$, while Arezzo--Della Vedova--Shi have proven an analytic analogue of Theorem \ref{introinterpolation}, producing constant scalar curvature K\"ahler metrics on suitable finite covers \cite{ADS}. The existence of constant scalar curvature K\"ahler metrics with cone angle singularities for $k \gg 0$ is an analogue of results of Hashimoto and Zeng on twisted constant scalar curvature K\"ahler metrics, and we rely on an algebro-geometric counterpart of these results proven by the first author and Ross \cite[Theorem 3.7]{stablemaps}. Arezzo--Della Vedova--Shi use these results to produce new examples of constant scalar curvature K\"ahler metrics, and we refer to their work for a discussion of examples to which these sorts of results can be applied \cite[Section 6]{ADS} (though we emphasise that applications of Theorem \ref{introinterpolation} are currently limited as we currently know relatively few examples of divisorially semistable varieties). 

In another direction, we note that Li has given examples of smooth polarised varieties which are uniformly K-stable on $\E^1$ using analytic techniques \cite[Proposition 6.12]{chi-li-csck} (hence divisorially stable by Boucksom--Jonsson); these results also apply to pairs, provided $X$ is smooth and $(X,(1-1/m)\Delta_X)$ is log canonical (as Li's method is insensitive to singularites of the divisor provided they are log canonical). Taking $\Delta_X$ to be singular, the $m$-fold branched cover $Y$ is singular and Theorem \ref{intromainthm} implies $(Y,L_Y)$ is $G$-equivariantly divisorially stable. This is a consequence of Li's work directly when $Y$ is smooth, but is inaccessible using analytic techniques when $Y$ is singular, meaning our result gives new examples of $G$-equivariantly stable varieties.

	\subsection*{Acknowledgements} We thank S\'ebastien Boucksom, Ivan Cheltsov, Joaqu\'in Moraga and R\'emi Reboulet for helpful comments. RD was funded by a Royal Society University Research Fellowship; TSP was funded by a postdoctoral fellowship associated with the aforementioned Royal Society University Research Fellowship.
	
	%	\subsection*{Notation} We work over an algebraically closed field $k$ of characteristic zero. Let $X$, $Y$ be irreducible projective varieties over $k$, i.e. integral projective $k$-schemes. Set $n \coloneqq \dim(X)$.

	\section{Divisorial stability of polarised varieties}\label{sec2}
	We work over an algebraically closed field $k$ of characteristic zero. We fix an $n$-dimensional normal projective variety $X$ along with an effective $\Q$-divisor $B$ such that $K_X+B$ is $\Q$-Cartier; we allow $B$ to be trivial. We also fix an ample $\Q$-line bundle $L$ on $X$.
	
	\subsection{Valuative stability}
	
	While we are ultimately interested in \emph{divisorial stability}, which is a notion of stability defined through \emph{convex combinations} of divisorial valuations on $X$, the general theory is quite intricate and simplifies considerably for a \emph{single} divisorial valuation. Thus we begin by explaining the theory for a single valuation, as introduced by the first author and Legendre \cite{dervan-legendre}, generalising the work of Fujita and Li in the Fano setting \cite{fujita, li}. A  reference explaining the background to the material presented here is \cite{lazarsfeld}.
	
	\begin{definition} A \emph{prime divisor over $X$} is an irreducible prime divisor $F \subset Y$ for some projective variety $Y$ which admits a birational morphism $\pi: Y \to X$. 
	
	\end{definition}
	
	A prime divisor $F$ over $X$ equivalently induces a valuation $\ord F$ on the function field $k(X)$ of $X$.
	
	\begin{definition} A \emph{divisorial valuation} is a valuation on $k(X)$ of the form $c \ord F$ for $F$ a prime divisor over $X$ and $c\in \R\geq 0$; we sometimes write this valuation as $v_{c\ord F}$.
	 \end{definition}
	
	By passing to a log resolution of singularities if necessary, we assume that the pair $(Y,F)$ is log smooth. To such a divisorial valuation we associate a numerical invariant called the \emph{beta invariant} of $F$, defined through the following standard invariants in birational geometry.
	
	\begin{definition} Suppose $L$ is a line bundle on $X$. We define the \emph{volume} of $L$ to be $$\Vol(L) = \limsup_{r\to\infty} \frac{\dim H^0(X,rL)}{r^n/n!},$$ and we say that $L$ is \emph{big} if $\Vol(L)>0$.\end{definition}
	
	The volume satisfies the homogeneity property $\Vol(lL) = l^n\Vol(L)$ and hence the definition extends to $\Q$-line bundles; it further extends to $\R$-line bundles by a continuity argument. We use two foundational results concerning the volume: firstly, the $\limsup$ involved in the definition is actually a limit; secondly, the volume is actually a continuously differentiable function on the cone of big ($\R$-)line bundles on $X$ \cite{BFJ}. We extend this definition to take $F$ into account by defining $$\Vol(L-xF) = \Vol(\pi^*L-xF),$$ where the latter is calculated on $Y$.
	
	We also require a measure of singularities, for which we use our hypothesis that $K_X+B$ is $\Q$-Cartier.
	
	\begin{definition} We define the \emph{log discrepancy} of $F$ to be $$A_{(X,B)}(F) = \ord_F(K_Y - \pi^*(K_X+B)) + 1.$$ \end{definition}
	
	Here we use that $(Y,F)$ is log smooth; if not, one should work on a log resolution of singularities of $(Y,F)$. The beta invariant is simply a combination of these invariants. Denote $$S_L(F) =  \frac{ \int_0^{\infty} \Vol(L-xF) dx}{L^n}.$$
	
	\begin{definition}\label{def:beta1}
		The \emph{beta invariant} of $F$ is defined to be $$\beta(F) = A_{(X,B)}(F) + \nabla_{K_X+B} S_L(F),$$where $$\nabla_{K_X+B}  S_L(F) = \frac{d}{dt}\bigg|_{t=0} S_{L+t(K_X+B)}(F).$$ 
	\end{definition}

	\begin{example} If $L=-K_X$ and $B=0$, so that $X$ is a Fano variety, this invariant reduces to $$\beta(F) = A_X(F)- S_{-K_X}(F),$$ which is precisely the invariant introduced by Fujita and Li \cite{fujita,li}. In general, our presentation of $\beta(F)$ agrees with that of Boucksom--Jonsson; see \cite[Theorem 2.18]{bj_non_arch_ii} for the equivalence with the original presentation  \cite{dervan-legendre}.
	
	\end{example}
	
%	Denote further $S(F) =  \int_0^{\infty} \Vol(L-xF) dx.$
	
	\begin{definition}We say that $((X,B);L)$ is
		\begin{enumerate}[(i)]
			\item \emph{valuatively semistable} if for all prime divisors $F$ over $X$ we have $\beta(F) \geq 0$;
			\item \emph{uniformly valuatively stable} if there exists an $\epsilon>0$ such that for all prime divisors $F$ over $X$ we have $\beta(F) \geq \epsilon S_L(F)$.
		\end{enumerate}
	\end{definition}
	
	Strictly speaking, in \cite{dervan-legendre} valuative stability required the divisors to be \emph{dreamy}, which is a finite-generation hypothesis. As this plays no role in the present work---and in light of the work of Boucksom--Jonsson on divisorial stability \cite{bj_non_arch_ii}, appears generally less relevant---we choose to omit this hypothesis (as in \cite{yaxiong-liu}).
	
	\begin{remark}\label{extension-of-discrepancy}
	These numerical invariants extend in a homogeneous way to divisorial valuations, namely by defining $A_{(X,B)}(a\ord F) = aA_{(X,B)}( F)$, $S_L(a\ord F) = aS_L(F)$ and $\beta(a\ord F) = a\beta( F)$. In this way, uniform valuative stability with respect to prime divisors over $X$ and with respect to divisorial valuations are equivalent. \end{remark}

	\subsection{Test configurations and K-stability}
	
	We next define test configurations and the associated Monge--Amp\`ere energy (which will be used in the subsequent sections) and uniform K-stability (which will play a secondary role to divisorial stability). 
	\begin{definition} A \emph{test configuration} for $((X,B);L)$ is a variety $\X$ along with:
		\begin{enumerate}[(i)]
			\item a $\Q$-Weil divisor $\B \subset \X$ and a $\Q$-line bundle $\L$;
			\item a $\C^*$-action on $\X$ fixing $\B$ and lifting to $\L$;
			\item a flat, $\mathbb G_m$-equivariant morphism $\pi: \X \to \mathbb A^1$ making $\B \to \mathbb A^1$  a flat morphism,
		\end{enumerate}
		such that each fibre $((\X_t,\B_t);\L_t)$ for $t\neq 0$ is isomorphic to $((X,B);L)$. We say that $((\X,\B);\L) $ is \emph{normal} if $\X$ is normal, \emph{ample} if $\L$ is relatively ample, \emph{semiample} if $\L$ is relatively semiample and \emph{nef} if $\L$ is relatively nef.
	\end{definition}
	
	The divisor $\B \subset \X$ is canonically defined by taking the $\mathbb G_m$-closure of $B \subset X \cong \X_1$. A test configuration admits a canonical compactification to a family over $\pr^1$ by compactifying trivially at infinity, and we will use the same notation for the resulting family $((\X,\B);\L) \to \pr^1$.
	
	\begin{definition} The \emph{Monge--Amp\`ere energy} of a nef test configuration $(\X,\L)$ is defined to be $$E(\X,\L) =  \frac{\L^{n+1}}{(n+1)L^n},$$ where this intersection number is calculated on the compactified test configuration over $\pr^1$. When we wish to emphasise the dependence on $L$, we denote this by $E_L(\X,\L)$.\end{definition}
	
	Note that this quantity is independent of $\B$ and hence we omit $\B$ from the notation. To define further numerical invariants, is is useful to pass to a resolution of indeterminacy of the natural rational map $\X \dashrightarrow X\times \pr^1;$ we thus obtain a new test configuration with an equivariant morphism to $X\times\pr^1$. We replace $\X$ by the associated resolution of indeterminacy, which we then say \emph{dominates the trivial test configuration} $(X\times \pr^1, L)$.
	
	\begin{definition} The \emph{minimum norm} of a nef test configuration $(\X,\L)$ is defined to be $$\|(\X,\L)\|_{\min} = \frac{\L^{n+1}}{(n+1)L^n} - \frac{\L^n.(\L-L)}{L^n},$$ where $L$ is pulled back to $\X$ through the morphism $\X \to X$. \end{definition}
	
	The minimum norm is called the ``non-Archimedean $I-J$-functional'' in \cite{boucksom_hisamoto_jonsson}; we follow the terminology of \cite{uniform}. This quantity is again independent of $\B$. The \emph{Mabuchi functional}, however, does actually depend on $\B$. In order to define this, for a test configuration dominating the trivial one define the \emph{entropy} $$\Ent(\X,\L) =  \frac{1}{L^n}\left(\L^n.K_{(\X,\B)/(X,B)\times\pr^1} +  \L^n.(\X_0 - \X_{0,\red})\right),$$ where $$K_{(\X,\B)/(X,B)\times\pr^1} = K_{\X}+\B - \pi^*(K_{X\times\pr^1} + B)$$ is the relative canonical class and $\X_0, \X_{0,\red}$ denote the central fibre of $\X$ and the reduced the central fibre respectively. 	To make sense of this definition, one uses that $\X$ is normal to ensure that $K_{\X}$ is a Weil divisor.

	\begin{definition}\label{def:mabuchi} We define the \emph{Mabuchi functional}  on the set of normal, nef test configurations to take the value $$M((\X,\B);\L) =\Ent(\X,\L) + \nabla_{K_X+B}E_L(\X,\L),$$ where  $((\X,\B);\L)$ is such a test configuration.\end{definition}

This is often called the \emph{non-Archimedean} Mabuchi functional; as its Archimedean counterpart plays no role in the present work, we simplify the terminology. It agrees with the more traditional  \emph{Donaldson--Futaki invariant}  of a normal, nef test configuration provided $\X_0$ is reduced; the associated notions of stability---which we next define---can be seen to be equivalent by a base-change argument \cite[Proposition 7.15]{boucksom_hisamoto_jonsson}. The directional derivative involved is defined by $$\nabla_{K_X+B}E_L (\X,\L) = \frac{d}{dt}\bigg|_{t=0} \frac{(\L+t(K_X+B))^{n+1}}{(n+1)(L+t(K_X+B))^n},$$ where we  assume (as we may) that the test configuration dominates the trivial one and $K_X+B$  is also then used to denote its pullback to $\X$; this derivative can be computed explicitly to produce a version of the Mabuchi functional more commonly used in the literature.

	\begin{definition} We say that $((X,B);L)$ is 
		\begin{enumerate}[(i)]
			\item \emph{K-semistable} if for all normal, nef test configurations $((\X,\B);\L)$ for $((X,B);L)$ we have $M(\X,\L) \geq 0$;
			\item \emph{uniformly K-stable} if there exists an $\epsilon>0$ such that for all normal, nef test configurations $((\X,\B);\L)$ for $((X,B);L)$ we have $M((\X,\B);\L) \geq \epsilon \|(\X,\L)\|_{\min}$.
		\end{enumerate}
	\end{definition}
	
	The relationship between uniform K-stability and K-stability is as follows: uniform K-stability with respect to test configurations \emph{with irreducible central fibre} is equivalent to uniform valuative stability \emph{with respect to dreamy prime divisors} \cite[Theorem 1.1]{dervan-legendre}. Roughly speaking, the central fibre of a test configuration with \emph{irreducible central fibre} induces a prime divisor over $X$, and conversely under a finite generation hypothesis the reverse of this construction also succeeds; the beta invariant is defined in such a way that it equals the value of the Mabuchi functional at the associated test configuration. To obtain stronger results---giving a fuller valuative interpretation of K-stability, in particular allowing test configurations with reducible central fibres---one needs further tools from non-Archimedean geometry, which we now turn to.
	
	\subsection{Berkovich analytification}
	%Including psh metrics. Extend energy to psh, and define finite energy. Explain want to compactify test configs. Mention here differentiability of energy.
	
	The appropriate way of viewing convex combinations of divisorial valuations is as a certain type of measure on the \emph{Berkovich analytification} $X\an$ of $X$, which we now define.  As throughout, we assume that $X$ is a normal  projective variety defined over an algebraically closed field of characteristic zero; the boundary divisor $B$ will be irrelevant in the present section. We refer to Reboulet \cite[Section 2]{reboulet} or Boucksom--Jonsson \cite{boucksom_jonsson_2022} for further details and proofs of the results stated below.
	
	For our purposes, it will be sufficient to define $X\an$ as a topological space; we note that it naturally carries the richer structure of a locally ringed space. We endow the field $k$ with the trivial absolute value.
	
	\begin{definition}
		As a set, we define the  \emph{Berkovich analytification} $X\an$ of $X$ to be the set of pairs $(V, |\cdot|_V)$, where $V$ is an irreducible subvariety of $X$ and $|\cdot|_V$ is an absolute value on the function field $k(V)$ extending the (trivial) absolute value on $k$. \end{definition}
	
	As a topological space, we first consider an affine chart $U \subset X$, where we define a topology on $U\an$ by requiring that for all $f\in \mathcal{O}_X(U)$ the evaluation map $$f:  U\an\to \mathbb{R}$$ defined by $ f(V, |\cdot|_V) = |f|_V$ be continuous. These topologies agree on intersections of affine charts of $X$, and hence glue to a topology on $X\an$ which is compact and Hausdorff. The association $X\to X\an$ is functorial, in the sense that a morphism $Y\to X$ of projective varieties induces a morphism of analytic spaces $Y\an\to X\an$.
	
 If we take $V=X$, then we simply obtain the function field of $X$. The space $X^{\val} \subset X\an$ of valuations on $X$ is then an open dense subset of $X\an$, and $X\an$  is thus a compactification of $X^{\val}$. We further denote $X^{\di}$ the space of divisorial valuations on $X$, so that $X^{\di }\subset X^{\val} \subset X\an.$

%  and the space the space of valuations of $X$. In fact, we can think of $X\an$ as the topological space whose points are \emph{semivaluations} on $X$, i.e. valuations $v: k(Y)^* \rightarrow \mathbb{R}$ for a subvariety $Y \subset X$. The space $X^{\val}$ is a dense subset of $X\an$, and $X\an$ can thus be viewed as a compactification of $X^{\val}$. 
	
	The   $\Q$-line $L$ bundle on $X$ induces a $\Q$-line bundle $L\an$ on $X\an$; rather than $L\an$ itself, what will be important is the space of \emph{non-Archimedean metrics} on $L\an$. We will give a shallow treatment of non-Archimedean metrics, omitting how to view Fubini--Study metrics as genuine metrics (in the sense of assigning a nonnegative number to a section at a point), and will instead view them as certain \emph{functions} on $X\an$. In K\"ahler geometry, the \emph{quotient} of two Hermitian metrics can be viewed as a function, and our treatment is reasonable due to the presence of the \emph{trivial metric} in the non-Archimedean setting of interest here. Here the trivial metric $\phi_{\triv}$ is induced by the trivial test configuration  $(X\times\mathbb A^1,L)$, with $X$ given the trivial $\mathbb G_m$-action.
	
	To define a function on $X\an$ associated to a test configuration, we may first assume that $(\X,\L)$ dominates the trivial test configuration by passing to a $\mathbb G_m$-equivariant resolution of indeterminacy of the rational map $X\times\A^1 \dashrightarrow \X$ if necessary. We can thus write $\L - L = D$ (with $L$ the pullback of $L$ from $X\times \A^1$ to $\X$) where $D$ is a $\Q$-Cartier divisor supported on $\X_0$. Writing $\X_0 = \sum_{j} b_j E_j$ as a cycle, so that the $E_j \subset \X_0$ are reduced and irreducible, this function is given as \begin{equation}\label{value-of-metric}\phi_{(\X,\L)}(v_{a\ord E_j}) = b_{j}^{-1}a\ord_{E_j}(D),\end{equation} with the function vanishing elsewhere. One checks that if $p: \X' \to \X$ is a $\mathbb G_m$-equivariant morphism with  $(\X',p^*\L)$ and	$(\X,\L)$ 	test configurations for $(X,L)$, then $\phi_{(\X',p^*\L)} = \phi_{(\X,\L)}$. We always identify $\phi_{(\X,\L)}$ and $\phi_{(\X',\L')}$ if the associated functions on $X\an$ are equal. Finally note that that $\phi_{(\X,\L)}$ is supported on $X^{\val}$.

%	We begin by defining \emph{Fubini--Study} metrics on $L\an$, which correspond to nef test configurations for $(X,L)$; after this, we define more general psh metrics on $L\an$. 
%	
%	
%	Thus consider a test configuration $(\X,\L)$ with $\L$ allowed to be relatively nef. We firstly define a \emph{reduction map} $$\red_{\X}: X\an \to \X_0$$ as follows. Any point $x = (\xi_x, |\cdot|_x)$ canonically induces a subscheme $\overline{\xi_x} \subset \X$ by taking the scheme-theoretic closure of $\xi_x$ under the $\mathbb G_m$-action; we define $\red_{\X}(x)$ to be the resriction of $\overline{\xi_x}$ to $\X_0$. Next, consider a section $s_U$ of $\L$ which is defined on a Zariski open neighbourhood of $\red_{\X}(x)$ and which does not vanish at $\red_{\X}(x)$. We then define a metric, which we denote $\phi_{(\X,\L)}$, on $L\an$ by requiring $$|s_U|_{\phi_{(\X,\L)}} = 1;$$ one checks that this is independent of choice of section $s_U$. 
	
	\begin{definition} 
		We define a \emph{Fubini--Study metric} to be a metric $\phi_{(\X,\L)}$ associated to a nef test configuration $(\X,\L)$. We denote by $\mathcal H^{\NA}(L\an)$ the set of Fubini--Study metrics on $L\an$.\end{definition}

When $L$ is clear from context we denote this simply by $\H^{\NA}$. We next define \emph{flag ideals} which allow us to obtain a more concrete picture on the relationship between Fubini--Study metrics and test configurations. Studied in \cite{odaka_ross_thomas, boucksom_hisamoto_jonsson}, we refer to \cite[Section 2.1]{boucksom_jonsson_2022} for further details.
	
	\begin{definition}\label{def:flag ideals}\cite[Section 2.1]{boucksom_jonsson_2022}
		We define a \emph{flag ideal} to be a vertical fractional ideal sheaf $\mathfrak{a}$ on $X\times \mathbb{A}^1$, i.e. a $\mathbb{G}_m$-invariant, coherent fractional ideal sheaf on $X\times \mathbb{A}^1$ that is trivial on $X\times \mathbb{G}_m$. 
	\end{definition}
	Here a vertical ideal sheaf by definition satisfies the condition that $\scO_{X\times \mathbb A^1}/\mfa$ is supported on $X\times \{0\}$. Fractional ideal sheaves are included as in the definition of a test configuration we allow $\L$ to be a $\Q$-line bundle. Any flag ideal admits a decomposition $$\mathfrak{a} = \sum_{\lambda\in \mathbb{Z}}t^{-\lambda} \mathfrak{a}_{\lambda}$$
	where $t$ is the coordinate of $\mathbb{A}^1$ and $\mathfrak{a}_{\lambda}\subset \mathcal{O}_X$ is a non-increasing sequence of integrally closed ideals on $X$ with $\mathfrak{a}_{\lambda} = 0$ for $\lambda \gg0$ and $\mathfrak{a}_{\lambda} = \mathcal{O}_X$ for $\lambda \ll0$. In addition, every test configuration $\X$ dominating the trivial one is given as the blow up $$\X = \operatorname{Bl}_{\mathfrak{a}}X\times\pr^1,$$ where with $D$ the exceptional divisor we define $\L\coloneqq L+D.$ Under this correspondence, the function $\phi_{(\X,\L)} = \phi_{\mfa}$ satisfies the property \cite[Equation 2.4]{boucksom_jonsson_2022} \begin{equation}\label{flag-function}\phi_{\mathfrak{a}\cdot\mathfrak{a}'} = \phi_{\mathfrak{a}} + \phi_{\mathfrak{a}'},\end{equation} where we use that the product of two flag ideals is a flag ideal.

	%In this way, there is a one-to-one correspondence between flag ideals and Fubini--Study metrics \cite[Theorem 2.7, Lemma 2.9]{boucksom_jonsson_2022}.
% 	
%	
%
%	
%	
%	and defining  $\phi_{\mathfrak{a}}\colon X\an\rightarrow \mathbb{R}$ by setting 
%	
%	
%	
%	 a
%	In addition, for any two flag ideals $\mathfrak{a}$, $\mathfrak{a}'$, we have 

	%A general \emph{continuous} metric is by definition of the form $\phi_{\triv} + \phi,$ where $\phi \in C^0(X\an)$, though continuous metrics will play no role in the present work. 
	
	% such metrics as functions on $X\an$ by setting $\phi_{(\X,\L)}(x)=.$ \color{red}As in \cite[\S 2]{boucksom_jonsson_2022}, the space $C^0(X)$ contains a dense subspace $\operatorname{PL}(X)$ of piecewise linear (PL) functions, where $\phi_D\in \operatorname{PL}(X)$ for a vertical $\mathbb{Q}$-Cartier divisor $D$ on some test configuration $\mathcal{X}\rightarrow\mathbb{A}^1$ for $X$, where $\phi_D(v) = \sigma(v)(D)$ for $v\in X\an$, where $\sigma \colon X\an\rightarrow \mathcal{X}\an$ is the Gauss extension. \color{black}
	Much as in K\"ahler geometry, it is also helpful to consider \emph{singular}  metrics. 
	
	\begin{definition}
		We define a \emph{plurisubharmonic metric} (or \emph{psh metric}) on $L\an$ to be a (pointwise) decreasing net of Fubini--Study metrics on $L\an$. 
	\end{definition}
	
	\begin{example}
	
	For two Fubini--Study metrics $\phi_{(\X,\L)}, \phi_{(\X',\L')}$, the condition $$ \phi_{(\X',\L')}\geq \phi_{(\X,\L)}$$means that we can find a $\mathbb G_m$-equivariant birational model $\Y$ with morphisms to both $\X,\X'$ such that on $\Y$ the difference $\L' - \L$ of pullbacks to $\Y$   is effective. 
	
	\end{example}
	
	\begin{remark}\label{psh as uniform limits}
		By \cite[Corollary 12.18 (iii)]{boucksom_jonsson_2022}, psh metrics on $L\an$ can also be viewed as (pointwise) decreasing limits of \emph{sequences} of Fubini--Study metrics on $X\an$, allowing the avoidance of nets.
	\end{remark}
	
	One should think that the theory of non-Archimedean geometry allows a language for discussing \emph{sequences} of test configurations and in particular for discussing \emph{compactness} properties for sequences of test configurations. 
	
	We will use the following extension of the Monge--Amp\`ere energy of a test configuration.
	
	\begin{definition}\cite[Sections 3, 7]{boucksom_jonsson_2022} We define the \emph{Monge--Amp\`ere energy} of a Fubini--Study metric $\phi_{(\X,\L)}$ to be $E(\phi_{(\X,\L)})= E(\X,\L)$. We extend this definition to arbitrary psh metrics $\psi$ by setting $$E(\psi) = \inf \{E(\phi_{(\X,\L)}): \phi_{(\X,\L)}\geq \psi\},$$ and define a psh metric to be of \emph{finite energy} if $E(\psi)>-\infty$. We denote by $\E^1(L\an)$ the space of finite energy psh metrics on $L\an$, or simply $\E^1$ when the polarisation $L$ is clear from context.
	\end{definition}
	
	We endow $\E^1(L\an)$ with the strong topology: the coarsest refinement of the weak topology (which requires convergence $\phi_j \to \phi$ if this holds pointwise) such that the Monge--Amp\`ere energy $E: \E^1(L\an)\to \R$ is continuous. With this topology, Fubini--Study metrics are dense in $\E^1(L\an)$.
	
	\begin{proposition}\cite[Proposition 7.7 (i)]{boucksom_jonsson_2022}\label{prop-MA-cont}
		The Monge--Amp\`ere energy is continuous along decreasing nets. In particular if $\phi_k$ is a sequence of Fubini--Study metrics decreasing to $\phi$, then $E(\phi_k)$ converges to $E(\phi)$.
		
	\end{proposition}%

\subsection{Measures on $X\an$}

As we have seen, test configurations are analogous to Fubini--Study metrics in non-Archimedean geometry. In K\"ahler geometry, it is beneficial to  consider volume forms and more general measures. The non-Archimedean construction of a measure associated to a metric is the following. Throughout,  if $v = c\ord_F$ is a divisorial valuation, viewed as an element of $X^{\di} \subset X\an$, we will denote by $\delta_{c\ord_F}$ the Dirac mass (or Dirac measure) supported  at $ v = c\ord_F$. 
	\begin{definition}\cite[Section 3.2]{boucksom_jonsson_2022} Denote by $\X_0 = \sum_j b_j E_j$ the central fibre of a test configuration $(\X,\L)$ as a cycle, so that the $E_j$ are reduced and irreducible. 
We define the \emph{Monge--Amp\`ere measure} $\MA(\phi_{(\X,\L)})$ of $\phi_{(\X,\L)}$  to be $$\MA(\phi_{(\X,\L)}) = \frac{1}{L^n}\sum_j b_j(\L^n.E_j) \delta_{b_j^{-1}\ord_{E_j}}.$$ 
		%where $\delta_{b_j^{-1}\ord_{E_j}}$ is the Dirac mass at the divisorial valuation $b_j^{-1}\ord_{E_j} \in X\an$. 
	\end{definition}

In the following, we denote by $\M$ the set of Radon probability measures on $X\an$, i.e. the dual space $C^0(X\an)^{\vee}$, which we endow with the weak topology.

\begin{proposition} \cite[Proposition 7.19 (iv)]{boucksom_jonsson_2022} There is a unique extension of the Monge--Amp\`ere measure from Fubini--Study metrics to general finite energy metrics defined in such a way that the map $\phi \to \MA(\phi)$ is continuous along decreasing nets.\end{proposition}

The inverse problem---associating a non-Archimedean metric to a measure---is the content of the \emph{non-Archimedean Calabi--Yau theorem} (originating in \cite{BFJ2}). We will require a general version of this, which involves \emph{finite norm measures}.

%We denote by $C^0(X)$ the Banach space of continuous functions $\phi\colon X\an \rightarrow \mathbb{R}$, endowed with the supnorm, and by $C^0(X)^{\vee}$ its topological dual, i.e. the space of (signed) Radon measures on $X\an$. It contains the subspace $\mathcal{M} = \mathcal{M}(X) \subset C^0(X)^{\vee}$ of Radon probability measures, which is convex and compact for the weak topology. As in \cite[\S 2]{boucksom_jonsson_2022}, the space $C^0(X)$ contains a dense subspace $\operatorname{PL}(X)$ of piecewise linear (PL) functions, where $\phi_D\in \operatorname{PL}(X)$ for a vertical $\mathbb{Q}$-Cartier divisor $D$ on some test configuration $\mathcal{X}\rightarrow\mathbb{A}^1$ for $X$, where $\phi_D(v) = \sigma(v)(D)$ for $v\in X\an$, where $\sigma \colon X\an\rightarrow \mathcal{X}\an$ is the Gauss extension. This construction is invariant under pullback to a higher test configuration, and one can thus always assume that $\mathcal{X}$ dominates the trivial test configuration $X \times \mathbb{A}^1$. 

\begin{definition}\cite[Definition 9.1]{boucksom_jonsson_2022} The \emph{norm} of a measure $\mu\in \mathcal{M}$ is defined as 
	$$\lVert \mu \rVert_{L}\coloneqq \sup_{\phi\in \mathcal{E}^1} \bigg\{ E(\phi) - \int_{X\an}\phi \dmu\bigg\}\in [0,+\infty].$$
	The space $\mathcal{M}^1 \subset\mathcal{M}$ of measures of \emph{finite norm} is defined as the set $$\mathcal{M}^1\coloneqq \{\mu\in\M \ | \ \lVert \mu\rVert_L < \infty\}.$$\end{definition}

The following then allows us to pass freely between measures and non-Archimedean metrics.

\begin{theorem}\cite[Theorem A]{boucksom_jonsson_2022}\label{BJ-CY} The Monge--Amp\`ere operator defines a bijection $$\E^1(L\an)/\R \to \mathcal{M}^1,$$ where $\E^1(L\an)/\R$ denotes finite energy metrics modulo the addition of constants. Furthermore, given a measure $\mu$, if $\MA(\phi) = \mu\in  \mathcal{M}^1$ and $\int_{X\an}\phi\dmu = 0$  then $$\|\mu\|_L = E(\phi).$$
\end{theorem}

\begin{remark}
	The supremum defining the norm of a measure can be taken over $\mathcal H^{\NA}(L\an)$; a benefit of considering the full space $ \mathcal{E}^1(L\an)$ is that by Theorem \ref{BJ-CY}, there is a $\phi \in \mathcal{E}^1(L\an)$ actually achieving the supremum. 
\end{remark}

The bijection $\E^1(L\an)/\R \to \mathcal{M}^1$---induced by the  non-Archimedean Calabi--Yau theorem---can be upgraded to a homeomorphism if $\mathcal{M}^1$ is given the strong topology, though this will not be used in the present work. We will, however, use a further differentiability result in associating numerical invariants to measures:

\begin{theorem}\cite[Theorem A]{bj_non_arch_ii}
	Fix a measure $\mu \in \M^1$ and denote $\Amp_{\Q}(X)$ the space of ample $\Q$-divisors modulo numerical equivalence. Then the function $\Amp_{\Q}(X) \to \R$ defined by $$L \to \|\mu\|_L$$ extends by continuity to a function on $\Amp_{\R}(X)$ which is  continuously differentiable.
\end{theorem}

For an $\R$-divisor $H$, we denote $$\nabla_{H} \|\mu\|_L\coloneqq \frac{d}{dt}\bigg|_{t=0} \|\mu\|_{L+tH}$$ the resulting directional derivative.

%For any $\phi\in \mathcal{E}^1$, the Monge–Amp{\`e}re measure $\operatorname{MA}(\phi)$ lies in $\mathcal{M}^1$, and $\phi$ computes its energy, i.e. $\lVert \operatorname{MA}(\phi) \rVert = E(\phi)- \int \phi \operatorname{MA}(\phi)$.

%The Monge–Amp{\`e}re operator defines a map $\operatorname{MA}\colon \mathcal{E}^1\rightarrow \mathcal{M}^1$. We define the normalized potential of a measure $\mu$ in the image of $\operatorname{MA}$ as the unique function $\phi_{\mu} \in \mathcal{E}^1$ such that
%$$\begin{cases}
	%    \mu = \operatorname{MA}(\phi_{\mu})& \\
	%    \int \mu\phi_{\mu}=0.& 
	%\end{cases}$$
	
	A more well-behaved subspace of $\M^1$ will be sufficient for our purposes.

	\begin{definition}\label{def: div measure}
		We define a \emph{divisorial measure} on $X\an$ to be a probability measure of the form $$\mu = \sum_j a_j\delta_{v_{j}}$$ for some finite collection $v_{j}$ of divisorial valuations on $X$, so in particular $\sum_{j=0}^m a_j = 1$. We denote by $\mathcal{M}^{\di}$ the set of divisorial measures. 
	\end{definition}

	%A divisorial measure $\mu\in \mathcal{M}^{\di}$ is thus a measure of the form $\mu \sum m_i\delta_{v_i}$, for a finite set of divisorial valuations $\delta_i$, where $\sum m_i = 1$.
	
	\begin{example}\label{ex:test-objects} Any divisorial valuation canonically induces a divisorial measure. Further, the Monge--Amp\`ere measure of any Fubini--Study metric is a divisorial measure. Thus divisorial measures can  be viewed as a simultaneous generalisation of divisorial valuations and  test configurations.\end{example}
	
\subsection{Uniform K-stability on $\E^1$}

With the construction of the Monge--Amp\`ere measure of a finite energy metric $\phi$ in hand, we may extend the uniform K-stability from test configurations to $\E^1$. Firstly, it is easily checked that the value taken by the Mabuchi functional at a test configuration depends only on the associated Fubini--Study metric (and similarly the minimum norm has the same property); we denote the resulting functional $$M: \H^{\NA} \to \R.$$ We extend the Mabuchi functional in the following manner.  For this, we firstly recall that  there is a natural way to extend the log discrepancy function $A_{(X,B)}: X^{\di} \to \R$ (which is nonnegative by definition if $(X,B)$ has at worst log canonical singularities) to a function $$A_{(X,B)}: X^{\NA} \to \R,$$ using semicontinuity of $A_X$ on $X^{\di}$ \cite[Definition A.2]{bj_non_arch_ii}.

	 Integration against the Monge--Amp\`ere measure associated to  $((\X,\B);\L)$  produces  \cite[Corollary 7.18]{boucksom_hisamoto_jonsson} $$\int_{X\an} A_{(X,B)} \MA(\phi_{(\X,\L)}) =\Ent(\X,\L);$$ we denote $$\Ent(\phi_{(\X,\L)}) = \Ent(\X,\L).$$ Secondly, the functional defined on $\H^{\NA}$ by $$\phi \to \nabla_{K_X+B} E_L(\phi)$$ extends in a continuous manner to $\E^1$, essentially because it contains only Monge--Amp\`ere energy (and ``mixed  Monge--Amp\`ere energy'') terms \cite[Theorem 7.14]{boucksom_jonsson_2022}, producing a natural extension to a functional $$M: \E^1 \to \R$$ taking the form $$M(\phi) = \Ent(\phi) + \nabla_{K_X+B} E_{L_X}(\phi).$$ The minimum norm similarly extends by continuity to a functional on $\E^1$.
	 
 \begin{definition}
 We say that $((X,B);L)$ is \emph{uniformly K-stable on $\E^1$} if there is an $\epsilon>0$ such that for all $\phi \in \E^1$ we have $$M(\phi) \geq \epsilon \|\phi\|_{\min}.$$
 \end{definition}	
	
	This condition is equivalent to Boucksom--Jonsson's notion of uniform K-stability with respect to filtrations and Li's notion of uniform K-stability with respect to models; see Li \cite[Section 2.1.3]{chi-li-csck} for a discussion and further details.

	\subsection{Divisorial stability} We are now in a position to associate numerical invariants to divisorial \emph{measures} (rather than metrics), and hence to define divisorial stability, following Boucksom--Jonsson \cite{bj_non_arch_ii}.  We begin with the entropy of $((X,B);L)$, which extends the log discrepancy of a single divisorial valuation to a general divisorial measure.

	\begin{definition}We define the \emph{entropy} $\operatorname{Ent}_{(X,B)}: \mathcal{M}^{\di} \rightarrow \mathbb{R}$  to be  
		$$\operatorname{Ent}_{(X,B)}(\mu) = \int_{X\an} A_{(X,B)}\dmu$$ where $\mu$ is a divisorial measure.\end{definition}
	
	%The above definitions allow us to define invariants as below.
	
	%\begin{definition}[{\cite[Definition 3.5]{bj_non_arch_ii}}]
	%    The \emph{$\delta$-invariant} of $(X,B;\omega)$ is defined as
	%    $$\delta(X,B;\omega) = \inf_{\mu\in \mathcal{M}^1\setminus\{\mu_{triv}\}} \frac{\operatorname{Ent}_{X,B}(\mu)}{\lVert \mu \rVert_{\omega}}.$$
	%\end{definition}
	%
	%
	%
	%For any $\theta \in N^1(X)$, we define a unique strongly continuous functional $\nabla_{\theta}\lVert\cdot\rVert_{\omega}\colon \mathcal{M}^1\rightarrow \mathbb{R} $ such that 
	%$$\nabla_{\theta}\lVert\operatorname{MA}(\phi) \rVert_{\omega} = \nabla_{\theta}E_{\omega}(\phi) = \frac{d}{dt}\big|_{t=0}E_{\omega+t\theta}(\phi).$$
	%Furthermore, we have 
	%$$\frac{d}{dt}\big|_{t=0}\lVert\mu \rVert_{\omega+t\theta} = \nabla_{\theta}\lVert\mu \rVert_{\omega}. $$
	
	Writing $\mu = \sum_j a_j\delta_{v_{j}}$, the entropy is given explicitly as $$\operatorname{Ent}(\mu) = \sum_j a_j A_{(X,B)}(v_j).$$ Note that the entropy is independent of the ample line bundle $L$. This allows us to define the \emph{beta invariant} of a divisorial measure on $((X,B);L)$.
	
	\begin{definition}{\cite[Definition 4.1]{bj_non_arch_ii}}
		The \emph{beta invariant} of a divisorial measure $\mu \in \mathcal{M}^{\di}$ is defined to be
		$$\beta_{((X,B);L)}(\mu)\coloneqq \operatorname{Ent}_{(X,B)}(\mu) + \nabla_{K_{X}+B}\lVert\mu \rVert_{L}.$$
	\end{definition}

%	\begin{example}\label{numerical} If instead $\mu = \MA(\phi_{(\X,\L)})$ is the Monge--Amp\`ere measure of a normal, nef test configuration  $(\X,\L)$ with reduced central fibre, then $\beta(\mu) = \DFu(\X,\L)$ (an error term appears when $\X_0$ is nonreduced that is straightforward to manage, where one must replace the Donaldson--Futaki invariant with the \emph{non-Archimedean Mabuchi functional}) and similarly $\|\mu\|_L = \|(\X,\L)\|_m$ \cite[Section 4.1]{bj_non_arch_ii}. 
%	\end{example}
	
	This allows us to define divisorial stability.

	\begin{definition}{\cite[Definition 4.3]{bj_non_arch_ii}}
		We say that $((X,B);L)$ is
		\begin{enumerate}[(i)]
			\item \emph{divisorially semistable} if for all divisorial measures $\mu$ on $X\an$ we have $\beta(\mu) \geq 0$ on $\mathcal{M}^{\di}$;
			\item \emph{divisorially stable} if there exists an $\epsilon>0$ such that for all divisorial measures $\mu$ on $X\an$ we have $\beta(\mu) \geq \epsilon \lVert \mu\rVert_L$.
		\end{enumerate}
	\end{definition}

	We may extend the beta invariant of a divisorial measure to a general finite norm measure in a way analogous to the extension of the Mabuchi functional to $\E^1$: the entropy is defined in the same way for divisorial and finite norm measures, while the norm itself remains differentiable (in the polarisation) for a general finite energy measure \cite[Theorem 2.15]{bj_non_arch_ii}, meaning we can define $$\beta(\mu) =  \operatorname{Ent}_{(X,B)}(\mu) + \nabla_{K_{X}+B}\lVert\mu \rVert_{L}$$ for $\mu \in \E^1$. The resulting notion of stability is equivalent to divisorial stability by continuity of the various quantities in the measure.
	
	\begin{example}\label{metrics-vs-measures} If $\phi \in \E^1$, then Boucksom--Jonsson prove the key equality  \cite[Equation (4.5)]{bj_non_arch_ii} $$M(\phi) = \beta(\MA(\phi)),$$ which implies that divisorial stability is \emph{equivalent} to uniform K-stability on $\E^1$ (using also the non-Archimedean Calabi--Yau theorem); in particular, divisorial stability \emph{implies} uniform K-stability. If instead $\mu = \delta_{v_F}$ for a divisorial valuation $v_F$ on $X$ associated to a prime divisor $F$ over $X$, then $\beta(\mu) = \beta(F)$ with $\beta(F)$ the $\beta$-invariant of Definition \ref{def:beta1}, and relatedly $S_L(F) = \|\mu\|_L$ \cite[Theorem 2.18]{bj_non_arch_ii}. Thus the $\beta$-invariant of finite norm measures (in particular divisorial measures) can be seen as a simultaneous generalisation of the $\beta$-invariant of divisorial valuations, and the Mabuchi functional on the set of Fubini--Study metrics (in particular test configurations). %	Similar results hold for the various semistability notions.

	\end{example}
 	
%	In particular by Example \ref{ex:test-objects} divisorial stability allows \emph{more} test objects than valuative stability and K-stability, and by Example \ref{numerical} clearly \emph{implies} uniform valuative stability and uniform K-stability. We quote the following result for context, and merely mention that one can generalise test configurations to \emph{filtrations}, and extend the definition of the Donaldson--Futaki invariant to them in a natural way. This produces a natural notion of uniform K-stability  \emph{with respect to filtrations}, which is also equivalent to Li's notion of \emph{uniform K-stability with respect to models} \cite{chi-li-csck} (see \cite{bj_non_arch_ii} for more details).
%	
%	
%	\begin{theorem}{\cite[Theorem D]{bj_non_arch_ii}}
%		Divisorial stability of $((X,B);L)$ is equivalent to uniform K-stability of $((X,B);L)$ with respect to filtrations. In particular divisorial stability implies uniform K-stability.\end{theorem}
%	
%	
	\subsection{Equivariant divisorial stability}
	
	Consider now a finite group $G$ acting on the projective variety $X$. Since $G$ acts on $X$, it acts on the function field of $X$ and hence on $X^{\di}$ by setting $(g(v))(f) = v(g^*f)$.

	\begin{definition}\label{G-invariant measures}
		We say that a divisorial measure $\mu = \sum_{j}  a_j\delta_{v_{j}}$ is $G$-\emph{invariant} if for all $g \in G$ $$\mu = \sum_{j} a_j\delta_{g(v_{j})}.$$  We denote the space of $G$-invariant divisorial measures by $\mathcal{M}^{\di, G}_Y$.
	\end{definition} 
	
	We will consider pushforwards of measures in Section \ref{finite-measures}, where we will show that this  condition asks $g_*\mu = \mu$ for all $g \in G$. 
	
	%	We similarly define $G$-invariant metrics on $X$. Consider an automorphism $g\colon X \rightarrow X$, induced by an element $g\in G$, and a metric $\phi$ on $X$. The pullback $g^* \phi$  of a psh metric $\phi$ is a well-defined psh metric on $X$ by \cite[Proposition 3.6]{boucksom_jonsson_2022}, whose construction will be 
	%	
	%		
	%	\begin{definition}\label{G-invariant metrics}
		%		We say that a metric $\phi$ on $X$ is $G$\emph{-invariant} if for all $g\in G$
		%		$$g^*\phi=\phi.$$
		%		We denote by $\E^{1, G}(L^{\an})$ the space of finite energy $G$-invariant metrics.
		%	\end{definition}
	
	% We also define a similar analogue for the energy of a measure on $X^{\an}$.
	
	% \begin{definition}The $G$\emph{-invariant energy} of a  measure $\mu\in \mathcal{M}$ is defined as 
		% 	$$\lVert \mu \rVert_{L}^G\coloneqq \sup_{\phi\in \mathcal{E}^{1,G}} \bigg\{ E(\phi) - \int\phi\mu\bigg\}\in [0,+\infty].$$\end{definition}
	
	% We are now in a position to define the $G$-invariant beta invariant. 
	
	% \begin{definition}{\cite[Definition 4.1]{bj_non_arch_ii}}
		% 	The \emph{$G$-invariant beta invariant} of a divisorial measure $\mu \in \mathcal{M}^{\di}$ is defined to be
		% 	$$\beta_{((X,B);L)}^G(\mu)\coloneqq \operatorname{Ent}_{(X,B)}(\mu) + \nabla_{K_{X}+B}\lVert\mu \rVert_{L}^G.$$
		% \end{definition}
	
	%The definition of a $G$-invariant beta invariant allows us 
	We are now in a position to introduce the notion of \emph{$G$-equivariant divisorial stability}.
	
	\begin{definition}\label{G-equivariant divisorial stability}
		We say that $((X,B);L)$ is
		\begin{enumerate}[(i)]
			\item \emph{$G$-equivariantly divisorially semistable} if for all $G$-invariant divisorial measures $\mu$ on $X\an$ we have $\beta(\mu) \geq 0$ on $\mathcal{M}^{\di}$;
			\item \emph{$G$-equivariantly divisorially stable} if there exists an $\epsilon>0$ such that for all $G$-invariant divisorial measures $\mu$ on $X\an$ we have $\beta(\mu) \geq \epsilon \lVert \mu\rVert$.
		\end{enumerate}
	\end{definition}
	
	To compare with equivariant notions of uniform K-stability, we first make the following definition.

	\begin{remark}
		In Section \ref{Sec:pushpullmetrics}, we define pullbacks of metrics under morphisms, which for $g: X \to X$ we denote $g^*\phi$. 
	%	With this, we may define $G$-invariant metrics by requiring that  $g^*\phi = \phi$ for all $g\in G$. We then denote by $\mathcal{E}^{1,G}$ the space of $G$-invariant finite energy psh metrics.
		 Furthermore, in Corollary \ref{energies of G invariant measures}, we show that for a $G$-invariant divisorial measure $\mu$  the sup defining the norm of $\mu$ can be taken over $G$-\emph{invariant} psh metrics, namely$$ \lVert \mu \rVert_{L}= \sup_{\phi\in \mathcal{E}^{1,G}} \bigg\{ E(\phi) - \int_{X\an}\phi\dmu\bigg\}$$ with $\E^{1,G}(L_X\an)$ the space of $G$-invariant finite energy metrics, giving some justification for the definition.
	\end{remark}

	\begin{remark}	We show in Theorem \ref{equiv-BJ} that $G$-equivariant divisorial stability is equivalent to uniform K-stability on $\E^{1,G}$, primarily using the work of Boucksom--Jonsson described in Example \ref{metrics-vs-measures} and some equivariant non-Archimedean geometry. We expect that, analogously to  the Fano case \cite{datar-szek, ZZ}, $G$-equivariant divisorial stability and divisorial stability are actually equivalent. \end{remark}
	
	\section{Divisorial stability under finite covers}
	
	Our next aim is to prove Theorem \ref{intromainthm}, explaining the behaviour of divisorial stability under finite covers. The level of generality of Theorem \ref{intromainthm} is an arbitrary Galois cover $\pi: Y \to X$ defined as the quotient under a group $G$, such that $L_Y = \pi^*L_X$ is ample and $$K_Y+\Delta_Y = \pi^*(K_X+\Delta_X)$$ for effective $\Q$-divisors $\Delta_Y, \Delta_X$. To ease notation, we prove this result in the notationally simpler case when $G$ is cyclic and $B$ is an irreducible $\Q$-divisor on $X$ such that by Riemann--Hurwitz $$K_Y = \pi^*\left(K_X +\left(1-\frac{1}{m}\right)B \right).$$ This is the most important special case for applications; the proof in the general case is identical, but requires an extra summation index at most steps.

	More precisely, our setup is the following. We take a normal projective variety $Y$ with a $G$-action, where $G$ is a finite cyclic group of degree $m$ and where $K_Y$ is $\Q$-Cartier. We let $X = Y/G$ be the quotient of $X$ by $G$, write $\pi: Y \to X$ for the resulting quotient map and let $B \subset X$ be the branch divisor. It follows that $K_X+\left(1-\frac{1}{m}\right)B$ is $\Q$-Cartier and satisfies $$K_Y = \pi^*\left(K_X +\left(1-\frac{1}{m}\right)B \right)$$  by Riemann--Hurwitz. We assume that $G$ lifts to an action on an ample $\Q$-line bundle $L_Y$ on $Y$ and let $L_X$ be its quotient, so that $\pi^*L_X = L_Y$. 
	
	\subsection{Finite maps between analytifications}
	
	The map $\pi: Y \to X$ induces a map $\pi\an: Y\an\to X\an$ defined by $$\pi\an(V, |\cdot|_V) =  (\pi(V), |\cdot|_{\pi(V)});$$ we begin by giving a more explicit, geometric description of this map on divisorial valuations. For a divisorial valuation $V$ is simply taken to be $Y$, so since $\pi$ is surjective, we obtain a valuation on $X$ from one on $Y$. Recall in general that the image of a valuation $v$ on $Y$ is defined for a rational function $f\in k(X)$ by setting $$\pi(v)(f) = v(\pi^*f).$$
	
	\begin{proposition}\label{divisorial_points_map_to_div_points}
		Let $u = c\ord_F\in Y^{\di}$ be a divisorial valuation on $Y$. Then $F$ can be realised as a prime divisor on a birational model $Y'$ of $Y$ such that $Y' \to Y$ is $G$-equivariant. Further, denoting $X' = Y'/G$ and denoting $D$ the image of $F$ in $X'$, then $\pi(u)$ is the divisorial valuation associated to $e_F c\ord_D,$ where  $e_F$ denotes the ramification index of $Y'\to X'$ along $F$.
	\end{proposition}
	
	%, hence divisorial points are mapped to divisorial points. 
	
	\begin{proof}

		We begin by replacing an arbitrary birational model $Y'$ of $Y$ with a birational model $Y'' \to Y$ which admits a lift of the $G$-action, in such a way that the morphism $Y'' \to Y$ is $G$-equivariant. It suffices to construct $Y''$ as a blowup of $Y$ along a $G$-invariant subscheme of $Y$, since in this case the $G$-action lifts automatically by the universal property of blowups. 
		
		Since $Y'$ is birational to $Y$, we may write $Y' = \Bl_{\mathcal{I}}Y$, where $\mathcal{I}$ is an ideal sheaf. We consider the orbit $$\mathcal{I}\cdot g^{-1}\mathcal{I} \cdot\ldots \cdot (g^{m-1})^{-1}\cdot \mathcal{I} \subset Y$$  of $\mathcal I$, which is a $G$-\emph{invariant} ideal sheaf (and where  $g^{-1}\mathcal I$ denotes the inverse image of $\mathcal I$ under $g$). Letting $Y'' \coloneqq \operatorname{Bl}_{\mathcal{I}\cdot g^{-1}\mathcal{I} \cdot\ldots \cdot (g^{m-1})^{-1}\cdot \mathcal{I}} Y$, by \cite[Corollary 1]{moody_2001} (namely, we use that the blowup of a product of ideal sheaves is the successive blowup of one factor and then the total transform of the other factor), we obtain birational morphisms $Y''\rightarrow Y' \rightarrow Y$, and by construction $Y''$ admits a $G$-action. Thus we replace $F\subset Y'$ with its proper transform on $Y''$, which does not modify the divisorial valuation itself.
		
		%Similarly, as in \cite{liu_zhu_2021}, starting from a birational model $X'\rightarrow X$ of $X$, we can define $Y'$ to be the normalization of the fiber product $Y \times_X X'$, where $X' = Y'/G$.
		
		As we now  assume $Y'$ admits a $G$-action making the morphism $Y'\to Y$ a $G$-equivariant morphism, we may  take the quotient $Y'/G$ of $Y'$ by $G$; we define $X'=Y'/G$. We then have a commutative diagram
		\begin{center}
			\begin{tikzcd}
				Y'\arrow[r, "\pi'"]\arrow[d] & X' \arrow[d]\\
				Y \arrow[r, "\pi"] & X,
			\end{tikzcd}
		\end{center} since $Y' \to X$ is $G$-invariant. 
		
		Setting $D = \pi'(F)$, it follows for example from \cite[Exercise 2.2]{silverman_2016} that $$\ord_F((\pi')^* f) =e_F \ord_D(f),$$
		where $e_F$ is the ramification index of $Y' \to X'$ along $F$ and $f \in k(X)$. This  completes the proof by the definition of the image of a valuation. \end{proof}

		\subsection{Pullbacks and pushforwards of metrics  under finite covers}\label{Sec:pushpullmetrics}

		Our next goal is  define pushforwards and pullbacks of metrics, in order to relate $G$-invariant Fubini--Study metrics on $Y$ to Fubini--Study metrics on $X$. 
		
		We recall the definition of the pullback of a psh metric, as defined by Boucksom--Jonsson \cite[Proposition 3.6]{boucksom_jonsson_2022}. We begin with a Fubini--Study metric $\phi$ on $L\an$ induced by a test configuration $(\X,\L_{\X})$ for $(X,L_X)$. The $\mathbb{G}_m$-equivariant rational map $Y\times \mathbb{A}^1 \dashrightarrow \X$ induced by $\pi$ admits a $\mathbb{G}_m$-equivariant resolution of indeterminacies, inducing a test configuration $(\mathcal{Y}, \L_{\Y})$ for $(Y,L_Y)$,  where $\L_{\Y}$ is the pullback of $\L_{\X}$ through the morphism $\Y \to \X$; this test configuration dominates $Y\times\mathbb A ^1$ by construction.

		\begin{definition} We define the \emph{pullback} of the Fubini--Study metric $\phi$ on $L_X\an$ associated to a test configuration $(\X,\L)$ to be the Fubini--Study metric on $L_Y\an$ induced by the test configuration $(\mathcal{Y}, \L_{\Y}).$  \end{definition}
		
		The pullback extends to arbitrary psh metrics by an approximation argument.

\begin{definition}\label{def:g-inv} We say that a psh metric $\phi$ on $L_X\an$ is \emph{$G$-invariant} if $g^*\phi = \phi$ for all $g \in G$. We denote by $\E^{1,G}(L_X\an)$ the space of $G$-invariant finite energy metrics.\end{definition}
		
		% In terms of Fubini--Study metrics, for a Fubini--Study metric $\phi = \phi_{(\X,\L_{\X})}$ corresponding to test configuration $(\X,\L_{\X})$, $\pi^* \phi$ is the Fubini--Study metric on $Y$, associated to the test configuration $(\Y,\L_{\Y})$, i.e. $\pi^* \phi = \phi_{(\Y,\L_{\Y})}$, where $(\Y,\L_{\Y})$ is the fibre product
		%		\begin{center}
			%			\begin{tikzcd}
				%				\Y \arrow[r]\arrow[d, "p_{\X}"] & Y\times \mathbb{P}^1 \arrow[d, "\pi"]\\
				%				\X \arrow[r] & X\times \mathbb{P}^1,
				%			\end{tikzcd}
			%		\end{center} 
		%		i.e. $\Y = \X \times_{X\times \mathbb{P}^1} Y \times \mathbb{P}^1$.
		We next define pushforwards of $G$-invariant Fubini--Study metrics in a similar spirit. Let $\phi$ be a $G$-invariant Fubini--Study metric $\phi$ on $L_Y\an$, which hence corresponds to a test configuration $(\Y,\L_{\Y})$ which can be taken to dominate the trivial test configuration. We will show in Proposition \ref{isomorphism between metrics of finite energy} that $(\Y,\L_{\Y})$ can be taken to be $G$-invariant, in the sense that it admits a $G$-action inducing the fixed action on $Y$ and commuting with the $\mathbb G_m$-action. Then by \cite[Lemma 3.1]{dervan_2016}, taking the quotient of  $(\Y,\L_{\Y})$ by the $G$-action induces a test configuration $(\X,\L_{\X})$ for $(X,L_X)$  such that $\pi^*\L_{\X} = \L_{\Y}$, where $\pi$ is the quotient map; by construction,  $(\X,\L_{\X})$ dominates the trivial test configuration provided $(\Y,\L_{\Y})$ dominates the trivial test configuration. 
		
		\begin{definition} We define the \emph{pushforward} of a $G$-invariant Fubini--Study metric $\phi$ on $L_Y\an$ corresponding to a $G$-invariant test configuration $(\Y,\L_{\Y})$  to be the Fubini--Study metric on $L_X\an$  associated to the test configuration obtained as the quotient of $(\Y,\L_{\Y})$ by $G$.\end{definition}

		We next prove that a $G$-invariant psh metric $\phi$ on $L_Y\an$ (which is a decreasing limit of Fubini--Study metrics) is  a decreasing limit of $G$-invariant Fubini--Study metrics  $\phi_k$, and that these  $G$-invariant Fubini--Study metrics can further be taken to be associated to $G$-invariant test configurations. %The lemma just proven thus implies that $\pi_* \phi_k$ is also a decreasing sequence of Fubini--Study metrics on $L_X\an$, producing a psh metric on $L_X\an$ which we define to be $\pi_*\phi$, producing a well-defined pushforward on arbitrary $G$-invariant psh metrics on $L_Y\an$.
		
		\begin{proposition}\label{isomorphism between metrics of finite energy}
			Every $G$-invariant psh metric can be realised as a decreasing limit of $G$-invariant Fubini--Study metrics on $L_Y\an$ associated to explicitly defined $G$-invariant test configurations.
			
		\end{proposition}
		\begin{proof}
			%We will first show that we may realise $\phi$ as a decreasing limit of $G$-invariant Fubini--Study metrics $\phi_k^G$, which are associated to test configurations $(\Y_k,\L_k)$ for $(Y,L_Y)$ which we may assume to dominate the trivial test configuration.    
			
			Let $\phi$ be a $G$-invariant psh metric on $Y\an$ with finite energy.	By Remark \ref{psh as uniform limits}, we may realise $\phi$ as a decreasing limit of Fubini--Study metrics $\phi_k$  associated to test configurations $(\Y_k,\L_k)$ for $(Y,L_Y)$ which we may assume dominate the trivial test configuration. We define 
			$$\phi_k^G = \sum_{g\in G}\frac{1}{|G|}g^*\phi_k$$
			which is by definition a $G$-invariant function. This is a convex combination of psh metrics, and is hence itself psh \cite[Theorem 4.7 (ii)]{boucksom_jonsson_2022}; we will see positivity more explicitly in what follows.
			
			We will first show that $\phi_k^G\to \phi$. Since $\phi$ is the decreasing limit of the Fubini--Study metrics $\phi_k$, for a point $y\in Y\an$, $\phi_k(g(y))$ decreases to $\phi(g(y)) = \phi(y)$, where we have used that $\phi$ is $G$-invariant. Hence, $\phi_k(y)^G=\phi_k^G(g(y))$ decreases to 
			$$\sum_{g\in G}\frac{1}{|G|}g^*\phi = \phi.$$
			
			We next give an explicit description of  $\phi_k^G$ as a Fubini--Study metric, for which we employ flag ideals. As $\phi_k$ is Fubini--Study, it corresponds to a nef test configuration $(\Y,\L_{\Y})$, and in addition $\Y = \Bl_{\mathfrak{a}}X\times\pr^1$	for a flag ideal $\mathfrak{a}$. Using a similar idea to the proof of Proposition \ref{divisorial_points_map_to_div_points}, we define $$\mathfrak{a}^G\coloneqq \mathfrak{a}\cdot (g^{-1}\mathfrak{a})\cdot \ldots \cdot ((g^{m-1})^{-1}\mathfrak{a}).$$  Notice that $\mathfrak{a}^G$ is a flag ideal on $Y\times\pr^1$, and set $$\Y^G = \Bl_{\mathfrak{a}^G} Y\times\pr^1,$$ which admits a $G$-action by construction. By \cite[Theorem 2.7, Proposition 3.6]{boucksom_jonsson_2022}, each $g^*\phi_k$ is Fubini--Study with associated flag ideal $g^*\phi_k =\phi_{g^{-1}\mathfrak{a}} $, where $g^{-1}\mfa$ denotes the inverse image of $\mfa$ under $g$. By Equation \eqref{flag-function}, the product of ideal sheaves  defining $\mathfrak{a}^G$ corresponds  precisely to the sum $$\phi_{\mathfrak{a}^G} = \sum_{g\in G}g^*\phi_k.$$ Thus $\phi_k^G$ is associated to the flag ideal $\mathfrak{a}^G$. 
			
			To understand the line bundle $\L^G$ on $\Y^G$ associated to the Fubini--Study metric $\phi_k^G$, note first that $\Y^G$ admits a morphism to the test configuration associated to $g^{j*}\phi_k$ for each $j$ in the same way as the proof of  Proposition \ref{divisorial_points_map_to_div_points}. The pullback metric $g^*\phi_{\L_k}$ can then be represented on $\Y^G$ itself through the pullback line bundle $g^*\L$ on $\Y^G$, since pullback of Fubini--Study metrics is defined through pulling back line bundles. Thus $\phi_k^G$ corresponds to the test configuration $(\Y^G, \L^G)$, where $$\L^G = \sum_{g\in G}\frac{1}{|G|}g^*\L;$$ note $\L^G$ is relatively nef as $\L$ is so, and similarly relatively semiample provided $\L$ is so (this can alternatively be obtained from  \cite[Proposition 2.25]{boucksom_jonsson_2022}). Since $\L$ can be viewed as a $G$-invariant $\Q$-Cartier divisor, it admits a lift of the $G$-action, meaning we have represented $\phi_k^G$ by a $G$-invariant test configuration $(\Y^G, \L^G)$. 	
			
			Thus any $G$-invariant psh metric is a decreasing limit of $G$-invariant Fubini--Study metrics induced by $G$-invariant test configurations, as claimed. 		\end{proof}

		We next relate pushforwards and pullbacks.
		
		\begin{proposition}\label{pushforward-pullback-prop}
			
			The  pushforward and pullback define an energy-preserving bijection between the set of $G$-invariant Fubini--Study metrics on $L_Y\an$ and Fubini--Study metrics on $L_X\an$. 
		\end{proposition}
		\begin{proof}
			
			Consider by Proposition \ref{isomorphism between metrics of finite energy} a  $G$-invariant Fubini--Study metric $\phi$ on $L_Y\an$ which has an associated $G$-invariant test configuration $(\Y,\L_{\Y})$, so that by definition of the pushforward $\pi_*\phi$ is associated to the quotient test configuration $(\X,\L_{\X}$). To prove that the pushforward and pullback are mutual inverses, it thus suffices to prove that $$\pi^*\pi_*\phi = \phi.$$  By definition, the pullback  $\pi^*\pi_*\phi$ is the $G$-invariant Fubini--Study metric on $Y\an$ corresponding to the $G$-invariant test configuration $(\Y',\L_{\Y'})$, where $\Y' = \X \times_{X\times \mathbb{P}^1} Y \times \mathbb{P}^1$ is the fibre product

			%			.  Each  such $G$-invariant  test configuration induces a test configuration for $(X,L_X)$ by taking the quotient by the $G$-action by  \cite[Lemma 3.1]{dervan_2016}, and in this way one obtains a non-Archimedean metric $\psi_k$ on $X\an$. 

			%			
			%			As such, moving forward, we will prove the result for $G$-invariant Fubini--Study metrics on $Y\an$, and the result will extend to $G$-invariant psh metrics on 
			%			$Y\an$ by the usual approximation argument.
			
			%	We claim that $\pi^*\pi_*\phi = \phi$. 
			\begin{center}
				\begin{tikzcd}
					\Y' \arrow[r]\arrow[d, "p_{\X}"] & Y\times \mathbb{P}^1 \arrow[d, "\pi"]\\
					\X \arrow[r] & X\times \mathbb{P}^1;
				\end{tikzcd}
			\end{center}  we set  $\L_{\Y'} = p_{\X}^*\L_{\X}$. 
			By the universal property of the fibre product $\Y'$, there is an induced morphism $\rho \colon \Y\rightarrow \Y'$, which then satisfies $\rho^*\L_{\Y'} = \L_{\Y}$. It follows that the non-Archimedean metrics associated to $(\Y,\L_{\Y})$ and $(\Y',\L_{\Y'})$ are equal, proving that pullback and pushforward induce a bijection between $G$-invariant Fubini--Study metrics on $L_Y\an$ and Fubini--Study metrics on $L_X\an$.

			We finally prove that this bijection is energy-preserving. Denote by $(\X,\L_{\X})$ a test configuration associated to $\phi_{(\X, \L_{\X})}$, and denote $(\Y,\L_{\Y})$ the test configuration associated to the pullback $\pi^*\phi_{(\X, \L_{\X})}$ defined through constructing an equivariant resolution of indeterminacy of $Y\times \mathbb A^1 \dashrightarrow \X$. Let $p: \Y \to \X$ be the resulting morphism. We calculate			\begin{equation*}
				\begin{split}
					E(\phi_{(\X, \L_{\X})}) &= \frac{(\L_{\X})^{n+1}}{(n+1)(L_X)^n}, \\ &= \frac{(p^*\L_{\X})^{n+1}}{(n+1)(\pi^*L_X)^n}, \\ &= \frac{(\L_{\Y})^{n+1}}{(n+1)(L_Y)^n},\\    
					&= E(\phi_{(\Y, \L_{\Y})}),
				\end{split}
			\end{equation*}
			which shows that $E(\phi_{(\X, \L_{\X})}) = E(\phi_{(\Y,\L_{\Y})})$. A similar calculation shows for a $G$-invariant Fubini--Study metric $\phi_{(\Y,\L_{\Y})}$ on $L_Y\an$ that $$E(\phi_{(\Y,\L_{\Y})}) = E(\pi_*\phi_{(\Y,\L_{\Y})}),$$ proving the result.	\end{proof}
		
		\begin{remark}\label{rmk:pullbackMA}
			By continuity of the Monge--Amp\`ere energy stated as Proposition \ref{prop-MA-cont}, the pullback of general psh metrics also preserves the Monge--Amp\`ere energy.
		\end{remark}

		%	hence, the construction is energy-preserving. Furthermore, since $\Y$ dominates $\Y'$, the induced metric $\pi^*\pi_*\phi$ on $Y\an$ is the same as $\phi$. 
		
		%In particular, the above constructions produce an explicit energy preserving bijection between $\mathcal{H}^{\operatorname{dom}, G}(L_Y)$ and $\mathcal{H}^{\operatorname{dom}}(L_X)$, where $\mathcal{H}^{\operatorname{dom}}(L_X)$ is the set of $L_X$-psh PL functions of the form $\phi_D$, with $D\in \operatorname{VCar}(\X)_{\Q}$ for a test configuration $\X$ dominating $\X_{\triv}$. One direction of the bijection is induced by fibre product construction, and the other direction is induced by taking the quotient of $\Y$ by $G$. The result then follows from \cite[Proposition 3.10]{boucksom_jonsson_2022} and the usual approximation argument.
		
		%Now Y^{\an} \to X^{\an} is G-invariant (right?), so we can define (\pi_*\phi)(u) = \phi(\pi^{-1}(u)), where although  \pi^{-1}(u) is a set, by G-invariance the value of \phi at any point in this set is the same. 
		\subsection{Pullbacks and pushforwards of measures under finite covers}\label{finite-measures}
		
		We next relate divisorial measures on $Y$ to those on $X$. We begin by recalling the explicit construction of divisorial valuations on $X$ from those on $Y$. Given a prime divisor $F\subset Y' \to Y$  over $Y$, by Proposition \ref{divisorial_points_map_to_div_points} we may assume that $Y'$ admits a $G$-action meaning we may form the quotient $X' = Y'/G$. We denote by $\pi(F)$ the prime divisor over $X$ given by the image of $F$ under the morphism $Y' \to X'$. Proposition \ref{divisorial_points_map_to_div_points} then shows that the image of the divisorial valuation $c\ord_F$ under the map $Y^{\val} \to X^{\val}$ is the divisorial valuation $e_F c\ord_{\pi(F)}$, where $e_F$ is the ramification index of $Y' \to X'$ along $F$.
		
Let $D = \pi(F)$ be a prime divisor over $Y$. Then $\pi^*D$, the pullback cycle, takes the form $$\pi^*D = \sum_{F_j\in \pi^{-1}(D)}e_{F_j}F_j,$$ with $e_{F_j}$ the ramification index along $F_j$. The ramification indices are equal for all such $F_j$, so in this expression $e_{F_j} = e_{F_l}$ for all $j,l$. In addition the divisors $F_j$ and $F_l$ belong to the same $G$-orbit, in the sense that for all $j,l$ there exists a $g\in G$ such that $F_j = g(F_l)$.

%		where the $e_{F_i}$ are the ramification indices of the divisors $F_i$. In this particular situation, we further have $e_{F_i} = e_{F_j}$ for all $F_i$ and $F_j$ such that $\pi(F_i) = \pi(F_j) = D$.
		
		%Notice, that for divisors $F_i$ over $Y$, the $F_i$ lie in the same orbit of the action of $G$ if and only if there exists a divisor $D$ over $X$ such that $\pi^*D = \sum_{i} e_{F_i}F_i$, where the $e_{F_i}$ are the ramification indices of the action, which in this case are all equal, i.e. $e_{F_i} = e_{F_j}$ for all $i$, $j$. 
		
		Restating Definition \ref{G-invariant measures} through the explicit interpretation of the image of a divisorial valuation, a divisorial measure $$\mu = \sum_i a_i \delta_{c_i\operatorname{ord}_{F_i}}$$ is $G$-invariant if $a_i=a_l$ and $c_i=c_l$ for all $i, l$ such that $F_i$ and $F_l$ lie in the same $G$-orbit, or equivalently such that  $\pi(F_i) = \pi(F_l) = D$ for a prime divisor $D$ over $X$. We will use the following notation for $G$-invariant divisorial measures on $Y\an$:
		
		\begin{equation}\label{notation-G}\mathcal{M}^{\di, G}_Y \ni \mu = \sum_{D/X}a_D \left( \sum_{F_j\in \pi^{-1}(D)}\delta_{c_D\operatorname{ord}_{F_j}}\right).\end{equation}
		Here, the first sum $\sum_{D/X}$ is taken over all prime divisors $D$ over $X$, and is finite since there is a finite number of non-zero $a_D$, while the second sum is taken over all divisors $F_j$ over $Y$ in the preimage of $D$. The coefficients $a_D$ are arbitrary nonnegative coefficients such that $\int_{Y\an} \dmu = 1$, so that the measure is a probability measure.

		We next consider pushforwards and pullbacks of measures. For a divsiorial measure, the pushforward  is given explicitly by the following expression.
		
		\begin{lemma}\label{explicit-pushforward}If $$\mu = \sum_i a_i\delta_{c_i\ord_{F_i}}\in \mathcal{M}^{\di}_Y$$ is a divisorial measure on $Y\an$, then $$\pi_*\mu = \sum_i a_i\delta_{e_{F_i}c_i\ord_{\pi(F_i)}}\in \mathcal{M}^{\di}_X$$  is a divisorial measure on $X\an$, where $e_{F_i}$ is the ramification index along $F_i$. \end{lemma}

	%	Following Definition , a divisorial measure $\mu = \sum a_i \delta_{\operatorname{ord}_{F_i}}$ is $G$-invariant if and only if $a_i = a_j$ for all $i$, $j$ such that $F_i$ and $F_j$ lie in the same orbit, i.e. $\pi(F_i) = \pi(F_j)$. 
		%\begin{definition}\label{pushforward measure}
		%	We define the \emph{pushforward} of a divisorial measure  $\mu\in \mathcal{M}^{\di}_Y$ by setting, for $U\subset X\an$ a Borel set,$$(\pi_*\mu)(U)\coloneqq \mu(\pi^{-1}(U)).$$
		%\end{definition}
			
		\begin{proof} For a single Dirac mass $ \delta_u$ supported at a point $u\in Y^{\di}$, for $U\subset X\an$ we have $$(\pi_*\delta_u)(U) = \delta_{\pi(u)}$$ by definition of the pushfoward. Thus $\pi_* \delta_u  \in \mathcal{M}^{\di}_X$ is a divisorial measure, since  Proposition \ref{divisorial_points_map_to_div_points} implies that $\pi(u)$ is itself a divisorial valuation. Writing $u = c\ord_F$, by Proposition \ref{divisorial_points_map_to_div_points} its image is given explicitly by $$\pi(c\ord_F) = e_{F}c\ord_{\pi(F)}.$$ The general case is identical. %One similarly checks that  $$\pi_*\mu = \sum_{i=0}^m a_i\delta_{\pi(u_i)}\in \mathcal{M}^{\di}_X.$$ 
		\end{proof}
		
		%
		%Notice that if $\mu =$, for some $u\in Y^{\di}$ we have 
		%\begin{equation*}
		%    \pi_*(\mu)(U) = \mu(\pi^{-1}(U)) =\begin{cases}
			%        1 & \text{if } u \in \pi^{-1}(U)\\
			%        0 & \text{otherwise}
			%    \end{cases} = \begin{cases}
			%        1 & \text{if } \pi(u) \in U\\
			%        0 & \text{otherwise}
			%    \end{cases} = \delta_{\pi(u)}(U) \in \mathcal{M}^{\di}_X. 
		%\end{equation*}

		While there is a canonical pushforward, we must define pullbacks explicitly.  For a single divisorial valuation $a\ord_D$, a divisorial valuation has image $a\ord_D$ in $X\an$ if and only if it takes the form $e_F^{-1}a\ord(F),$ where $\pi(F) = D$ and $e_F$ is the ramification index along $F$, since $$\pi(e_F^{-1}c\ord(F)) = c\ord_D$$  by Proposition \ref{divisorial_points_map_to_div_points}. Writing $$\pi^*D = \sum_{F_j\in \pi^{-1}(D)}e_{F} F_j,$$ where as before the ramification indices $e_F \coloneqq e_{F_j}$ are equal for each $j$, we  define $$\pi^*\delta_{a\ord_D} = \sum_{F_j\in \pi^{-1}(D)}\frac{e_F}{|G|}\delta_{e_F^{-1}c\ord(F_j)}.$$  Note that $\pi^*\delta_{a\ord_D}$  is still a probability measure, since by the orbit-stabiliser theorem $$\sum_{F_j\in \pi^{-1}(D)} \frac{e_F}{|G|} = 1.$$We now define pullbacks of general divisorial measures in a similar way, essentially extending linearly.

		%\color{red} Theo, could you please fix the below to correct $G$-invariance? Also please define the ``normalisation parameters'' more explicitly, there are lots of numbers that sum to one, do you mean that you should divide by the size of the orbit? Also please fix the statement of Prop 3.6, the statement is that the pushforward/pullback are inverses and hence define an isomorphism. Could you also add an example---like Proposition 3.1---which explains geometrically what the pullback of a single Dirac delta at a divisorial valuation is? I think this might be in Liu-Zhu, maybe. But it seems to be implicit in the below but it should be explicit.

		%\color{black}
		
%		
%		
%		\begin{definition}\label{pullback measure}    
%			We define the \emph{pullback} of a divisorial measure $$\mu =\sum_i  a_i\delta_{c_i\ord_{D_i}} \in \mathcal{M}^{\di}_X$$ by 
%			$$\pi^*\mu = \sum_{i}a_i \sum_{u_{j} \in \pi^{-1}(c_i\ord_{D_i})}\frac{e_{F_i}}{|G|} \delta_{u_j},$$ where:
%			\begin{enumerate}[(i)]
%			\item the $u_j$ are divisorial valuations in $Y\an$ such that  $\pi(u_{j}) = c_i\ord_{D_i}$;
%			\item the pullback of $D_i$ is given as $$\pi^*D_i = \sum_j e_{F_i}F_{i,j};$$
%%			\item the normalisation coefficients are given by $$e_i= \frac{e_{F_{i,j}}}{|G|}.$$
%			\end{enumerate}
%		\end{definition}
%		
				\begin{definition}\label{pullback measure}
			We define the \emph{pullback} of a divisorial measure $$\nu =\sum_{D/X}  a_D\delta_{c_D\ord_{D}} \in \mathcal{M}^{\di}_X$$ by 
			$$\pi^*\nu = \sum_{D/X}a_D \left(\sum_{F_j \in \pi^{-1}(D)}\frac{e_{F_j}}{|G|} \delta_{e_{F_j}^{-1}c_D\ord(F_j)}\right).$$
%			\begin{enumerate}[(i)]
%			\item the $u_j$ are divisorial valuations in $Y\an$ such that  $\pi(u_{j}) = c_i\ord_{D_i}$;
%			\item the pullback of $D_i$ is given as $$\pi^*D_i = \sum_j e_{F_i}F_{i,j};$$
%%			\item the normalisation coefficients are given by $$e_i= \frac{e_{F_{i,j}}}{|G|}.$$
			%\end{enumerate}
		\end{definition}

		As before, the sum $D/X$ denotes a finite sum of prime divisors over $X$. As in the case when $\mu$ is a Dirac mass at a single divisorial valuation, it follows from the  orbit-stabiliser theorem that $\pi_*\mu$ is a probability measure.

		\begin{proposition}\label{isomorphism between divisorial measures}
			The above pushforward and pullback constructions are mutual inverses between the spaces of $G$-invariant divisorial measures on $Y\an$ and divisorial measures on $X\an$.
		\end{proposition}
		
		In particular pushforward and pullback induce an isomorphism $$\mathcal{M}^{\di, G}_Y \cong \mathcal{M}^{\di}_X.$$
		\begin{proof}
			We consider a $G$-invariant divisorial measure  $\mu\in \mathcal{M}^{\di, G}_Y$ and begin by showing that $$\pi^*\pi_*\mu = \mu.$$ Continuing the notation used in Equation \eqref{notation-G}, we denote $$\mu = \sum_{D/X}a_D\left( \sum_{F_j\in \pi^{-1}(D)}\delta_{c_D\operatorname{ord}_{F_j}}\right),$$
			where the first sum is taken over a finite sum of prime divisors $D$ over $X$ and the $a_D$ are coefficients such that $\sum_{D/X} a_D = 1$, and the second sum is taken over all divisors $F_j$ in the preimage of $D$. 
			
			We calculate
			\begin{equation*}
				\begin{split}
					\pi_*\mu &= \sum_{D/X} a_D \left( \sum_{F_j\in \pi^{-1}(D)} \delta_{e_{F_j}c_D\operatorname{ord}_D}\right) \\
					&= \sum_{D/X} a_D \delta_{e_{F_j}c_D\operatorname{ord}_D}\left(\sum_{F_j\in \pi^{-1}(D)} 1\right)
				\end{split}
			\end{equation*}
			where  $e_{F_j}$ is the common ramification index of the $F_j$, so $\pi^*D = e_{F_j} \sum_j F_j$. Then,
			\begin{equation*}
				\begin{split}
					\pi^*\pi_*\mu &= \sum_{D/X} a_D \left( \sum_{F_l\in \pi^{-1}(D)}\delta_{c_D\ord_{F_l}}\frac{e_{F_l}}{|G|}\left(\sum_{F_j\in \pi^{-1} (D)}1\right)\right),  \\
					&= \sum_{D/X} a_D \sum_{F_l\in \pi^{-1}(D)}\delta_{c_D\ord_{F_l}},
				\end{split}
			\end{equation*}
			where we use that $$\frac{e_{F_l}}{|G|}\sum_{F_j\in \pi^{-1}(D)}  1 =1$$ by the orbit-stabiliser theorem and the fact that $e_{F_l}=e_{F_j}$ for all $l,j$. Thus $$\pi^*\pi_*\mu = \mu,$$ as claimed.

%			
%			From Proposition \ref{divisorial_points_map_to_div_points}, we know that $\pi(u_i) = e_{F_i}\ord_D$ if and only if $u_i = \ord_{F_i}$, and $\pi(F_i) = D$. Hence, 
%			$$\pi^*\pi_*\mu = \sum_{D/X} a_D \sum_{F_i\in \pi^{-1}(D)}\delta_{\ord_{F_i}} = \mu.$$
			
			In the reverse direction, let  $\nu\in \mathcal{M}^{\di}_X$ take the form $$\nu =\sum_{D/X}  a_D\delta_{c_D\ord_{D}},$$ so that by definition $$\pi^*\nu = \sum_{D/X}a_D \left(\sum_{F_j \in \pi^{-1}(D)}\frac{e_{F_j}}{|G|} \delta_{e_{F_j}^{-1}c_D\ord_{F_j}}\right).$$ The pushforward of this measure is then given by
			
			\begin{equation*}
				\begin{split}
					\pi_*\pi^*\nu &= \sum_{D/X}a_D \left(\sum_{F_j \in \pi^{-1}(D)} \frac{e_{F_j}}{|G|}\delta_{c_D\ord_D}\right),\\
					& = \sum_{D/X}a_D \delta_{c_D\ord_D}\sum_{F_j \in \pi^{-1}(D)} \frac{e_{F_j}}{|G|},\\
					& =  \sum_{D/X}a_D \delta_{c_D\ord_D}, \\
					&=\nu,
				\end{split}    
			\end{equation*} where we again use the orbit-stabiliser theorem. This completes the proof. 					\end{proof}

	%		$$\nu = \sum a_i\delta_{v_i}$$ for some $v_i\in X^{\di}$, and $\pi^*\mu = \sum_{i}a_i \sum_{\pi(u_j) = v_i} e_{u_j}^i\delta_{u_j}$. 

		We end this section by showing that $G$-invariant measures of finite norm are given as Monge--Amp{\`e}re measures of $G$-invariant psh metrics, when $G$ is a finite group.
		
		\begin{proposition}\label{G-invarinant measures and energy comparison}
			Let $G \subset \Aut(Y,L_Y)$ be a finite group, and suppose $\mu\in \mathcal{M}^{\di,G}_Y$ is a $G$-invariant divisorial measure. Then the solution $\phi\in \E^1(L_Y\an)$ of the Monge--Amp\`ere equation $$		\MA(\phi) = \mu$$ is a $G$-invariant metric.		\end{proposition}
		\begin{proof}
			%TODO
			
			Letting $g \in G \subset \Aut(Y,L_Y)$ be such that $g_*\mu=\mu$, we must show that $g^*\phi = \phi$. We firstly claim that it is enough to show that \begin{equation}\label{MA-eqn-proof}\MA(\phi) = g_*\MA(g^*\phi).\end{equation} Indeed, by Proposition \ref{isomorphism between divisorial measures} $$g^*g_*\MA(g^*\phi) = \MA(g^*\phi),$$ so by Equation \eqref{MA-eqn-proof} \begin{align*} \MA(g^*\phi) &= g^* \MA(\phi), \\ &= g^*\mu, \\ &=\mu,\end{align*} where we use that $\mu$ is $G$-invariant.  Thus $\phi$ and $g^*\phi$ both solve the Monge--Amp{\`e}re equation for the measure $\mu$, meaning they must be equal \emph{up to the addition of a constant}, by uniqueness of solutions of the Monge--Amp{\`e}re equation, namely Theorem \ref{BJ-CY}. Since, for example, $\phi$ and $g^*\phi$ have the same supremum, they must genuinely be equal.
			
It therefore suffices to prove that for $g \in G \subset \Aut(Y,L_Y)$ $$\MA(\phi) = g_*\MA(g^*\phi).$$	By Remark \ref{psh as uniform limits}, we may realise $\phi\in \E^1(L_Y\an)$ as a decreasing limit of Fubini--Study metrics $\phi_k$,  associated to test configurations
			$(\Y_k, \L_{\Y_k})$ which we may assume dominate the trivial test configuration. As such we will prove the equality of Equation \eqref{MA-eqn-proof}  for Fubini--Study metrics, and the result will follow in general by continuity of the Monge--Amp{\`e}re operator.
			
			Thus let $(\Y, \L_{\Y})$ be a test configuration on $(Y, L_Y)$ dominating the trivial test configuration, associated to a Fubini--Study metric $ \psi_{(\Y, \L_{\Y})}$, and denote the  central fibre of $\Y$ by
			$$\Y_0 = \sum_jb_jE_j,$$
			where the $E_j$ are reduced and irreducible. By definition of the Monge--Amp\`ere operator 
			$$\MA(\psi_{(\Y, \L_{\Y})}) = \frac{1}{L_Y^n}\sum_j b_j(\L^n_{\Y}.E_j) \delta_{b_j^{-1}\ord_{E_j}},$$ where $\delta_{b_j^{-1}\ord_{E_j}}$ is the Dirac mass at the divisorial valuation $b_j^{-1}\ord_{E_j} \in Y^{\di}\subset Y\an$.
			
			The morphism $g$ induces by pullback a test configuration $(\Y', \L_{\Y'})$ for $(Y,L_Y)$, where by Proposition \ref{isomorphism between metrics of finite energy}  we may  (and do) assume that $\Y'\cong \Y$, but where the line bundle $\L_{\Y'} = g^* \L_{\Y}$ may not agree with $\L_{\Y}$ (indeed, we wish to show $g^* \psi_{(\Y, \L_{\Y})} =  \psi_{(\Y, \L_{\Y})}$, which precisely asks that $g^* \L_{\Y} = \L_{\Y})$. Since $g: \Y \to \Y$ is an isomorphism, for each $j$ the pullback $g^{*}E_j$  (which is simply the preimage $g^{-1}(E_j)$) is a single irreducible component of $\Y_0$, and further if  $g^{*}E_j = E_l$ then $b_j=b_l$; we use these properties frequently  in what follows.

			%$\Y'$ is the fibre product constructed in Section \ref{Sec:pushpullmetrics} and used in the definition of the pullback. 

						The push-pull formula gives $$\L_{\Y'}^n.g^{*}E_j = \L_{\Y}^n.E_j,$$ which in turn	  produces
				\begin{align*}
					\MA(g^*\psi_{(\Y, \L_{\Y})}) &= \frac{1}{L_Y^n}\sum_j b_j(\L^n_{\Y'}.g^{*}E_j) \delta_{b_j^{-1} \ord_{g^*E_j}}, \\ &= \frac{1}{L_Y^n}\sum_j b_j(\L^n_{\Y}.E_j) \delta_{b_j^{-1} \ord_{g^{*}E_j}}. \end{align*}
			By Lemma \ref{explicit-pushforward}, using that $g$ is an $g: \Y \to \Y$ is an isomorphism and hence is unramified,
			\begin{equation*}
				\begin{split}
					g_*\MA(g^*\psi_{(\Y, \L_{\Y})})&=\frac{1}{L_Y^n}\sum_j b_j(\L^n_{\Y}.E_j) \delta_{b_j^{-1}\ord_{g(g^{*}(E_j))}},\\
					&= \frac{1}{L_Y^n}\sum_j b_j(\L^n_{\Y}.E_j) \delta_{b_j^{-1}\ord_{E_j}},\\
					&=\MA(\psi_{(\Y, \L_{\Y})}),
				\end{split}
			\end{equation*}
			as required. %The proof is then concluded from \cite[Proposition 3.6]{boucksom_jonsson_2022}.
		\end{proof}
		
		%		We apply this to $G$ the finite group acting on $(Y,L_Y)$, but the result applies more generally. 
		A consequence, using the fact that the norm of a measure is computed as the energy of its Monge--Amp\`ere-inverse by Theorem \ref{BJ-CY}, is the following:
		
		\begin{corollary}\label{energies of G invariant measures}
			Let $\mu \in \mathcal{M}^{\di,G}_Y$ be a $G$-invariant divisorial measure on $Y\an$. Then
			$$\lVert \mu \rVert_{L}= \sup_{\phi\in \mathcal{E}^{1,G}} \bigg\{ E(\phi) - \int_{Y\an}\phi\dmu\bigg\}.$$
		\end{corollary}
		
		\subsection{Proof of the main theorem}
		We next turn to comparing numerical invariants of divisorial measures under finite covers, and in particular to the proof of Theorem \ref{intromainthm}. We begin with the energies of the measures, and throughout this section we include subscripts in various numerical invariants to emphasise whether we are calculating quantities on $X$ or on $Y$. 
		
		\begin{proposition}\label{equivalence of energies}
			%Let $-K_{X,D} = -K_X - (1-\frac{1}{m})D$. 
			Let $\mu \in \mathcal{M}^{\di, G}_Y$ be a divisorial measure on $X$. Then 
			\begin{equation*}
				\lVert \mu \rVert_{L_Y} = \lVert \pi_*\mu \rVert_{L_X}.
			\end{equation*}
		\end{proposition}
		\begin{proof}
			By Corollary \ref{energies of G invariant measures}, we must show that $$\sup_{\phi\in \mathcal{E}_Y^{1,G}} \bigg\{ E(\phi) - \int_{Y\an}\phi\dmu\bigg\} = \sup_{\psi\in \mathcal{E}_X^{1}} \bigg\{ E(\psi) - \int_{X\an}\psi\pi_*(\dmu)\bigg\}.$$ We may take the supremum defining $\lVert \mu \rVert_{L_Y} $ with respect to $G$-invariant Fubini--Study metrics on $L_Y\an$ by Proposition \ref{isomorphism between metrics of finite energy} and the continuity property of the Monge--Amp\`ere energy stated in Proposition \ref{prop-MA-cont}. Similarly the norm $\lVert \pi_*\mu \rVert_{L_X}$ can be computed as a supremum over Fubini--Study metrics on $L_X\an$.
			
			Taking an arbitrary $\psi \in\mathcal H^{\NA}(L_X\an)$, by Proposition \ref{pushforward-pullback-prop} its pullback $\pi^*\psi$ is a $G$-invariant Fubini--Study metric satisfying $E(\psi) = E(\pi^*\psi)$; Proposition \ref{pushforward-pullback-prop} also implies any $G$-invariant Fubini--Study metric $\phi\in\mathcal H^{\NA}(L_Y\an)$ can be realised in this way. Then $$ \int_{Y\an}\pi^*{\psi}\dmu = \int_{X\an} \psi \pi_*(\dmu)$$ by definition of the pushforward measure, completing the proof.  \end{proof}

		\begin{corollary}\label{equivalence of energies reverse}
			Let $\nu \in \mathcal{M}^{\di}_X$ be a divisorial measure on $Y$. Then 
			\begin{equation*}
				\lVert \nu \rVert_{L_X} = \lVert \pi^*\nu \rVert_{L_Y}
			\end{equation*}
		\end{corollary}
		\begin{proof}
			From Proposition \ref{isomorphism between divisorial measures} the pullback $ \pi^*\nu \in \mathcal{M}^{\di,G}_Y$ satisfies $\pi_*\pi^*\nu = \nu$, so by Proposition \ref{equivalence of energies} 
			\begin{align*}
				\lVert \pi^*\nu \rVert_{L_Y} &= \lVert \pi_*\pi^*\nu \rVert_{L_X}, \\ &= \lVert\nu \rVert_{L_X},
			\end{align*} as required. \end{proof}
		
		The following is a trivial consequence. 
		
		\begin{corollary}\label{equivalence of differentials}
			Let $H$  be an $\R$-divisor on $X$. Then for $\mu \in \mathcal{M}^{\di}_X$
			$$ \nabla_{H}\lVert\mu \rVert_{L_X} = \nabla_{\pi^*H}\lVert\pi^*\mu \rVert_{\pi^*L_X} .$$
			Similarly for $\nu \in \mathcal{M}^{\di,G}_Y$
			$$ \nabla_{H}\lVert \pi_*\nu \rVert_{L_X} = \nabla_{\pi^*H}\lVert\nu \rVert_{\pi^*L_X} .$$
		\end{corollary}
		%\begin{proof}
		%Let $\omega \in \operatorname{Amp}(X)$, $\theta \in N^1(X)$, and $\mu \in \mathcal{M}^{\di}_X$, $\nu \in \mathcal{M}^{\di, G}_Y$.
		%    Then, 
		%    \begin{equation*}
			%       \nabla_{\theta}\lVert\mu \rVert_{\omega} = \frac{d}{dt}\big|_{t=0}\lVert\mu \rVert_{\omega+t\theta} = \frac{d}{dt}\big|_{t=0}\lVert\pi^*\mu \rVert_{\pi^*(\omega+t\theta)} = \nabla_{\pi^*\theta}\lVert\pi^*\mu \rVert_{\pi^*\omega},
			%    \end{equation*}
		%and 
		%    \begin{equation*}
			%       \nabla_{\theta}\lVert \pi_*\nu \rVert_{\omega} = \frac{d}{dt}\big|_{t=0}\lVert\pi_*\nu \rVert_{\omega+t\theta} = \frac{d}{dt}\big|_{t=0}\lVert\nu \rVert_{\pi^*(\omega+t\theta)} = \nabla_{\pi^*\theta}\lVert\nu \rVert_{\pi^*\omega}.
			%    \end{equation*}
		%
		%    
		%\end{proof}
		
		We are now in a position to prove our main result. We recall the setup, which is a cyclic Galois cover $\pi\colon (Y,L_Y) \rightarrow (X, L_X)$ of normal projective varieties of degree $m$ with branch divisor $B \subset X$. We assume  $\pi^*L_X = L_Y$ and that $K_X+(1-1/m)B$ is $\Q$-Cartier, so that by Riemann--Hurwitz $$K_Y = \pi^*(K_X+(1-1/m)B).$$
		
		\begin{theorem}\label{main thm}
			$(Y,L_Y)$ is $G$-equivariantly divisorially (semi-)stable if and only if $((X,(1-1/m)B);L_X)$ is divisorially (semi-)stable.
			
		\end{theorem}
		
		\begin{proof}
			
			We prove the result for semistability, the proof for stability is identical. First, supposing that $((X,B);L_X)$ is divisorially semistable, we aim to show that $\beta(\mu)\geq 0$ for any $\mu \in \mathcal{M}^{\di, G}_Y$. %Let $(\mathcal{Y}, \mathcal{L}_Y)$ be a $G$-invariant test configuration for $(Y,L_Y)$ as in the proof of Proposition \ref{equivalence of energies}. 
			% It was shown in \cite[Lemma 3.1]{dervan_2016} that $(\mathcal{X}\coloneqq \mathcal{Y}/G, \mathcal{L}_X)$ is a $G$-invariant test configuration for $((X,D);L_X)$. 
			%From Proposition \ref{equivalence of energies}
			%we know that $ \lVert \pi^*\mu \rVert_{L_Y-tK_Y} = \lVert \mu \rVert_{L_X - tK_{X,(1-\frac{1}{m})D}}$. 
			As in the proof of Proposition \ref{divisorial_points_map_to_div_points}, we may assume that each of the divisors comprising $\mu$ lives on a model $Y'$ of $Y$ which admits a $G$-action making $Y' \to Y$ a $G$-equivariant morphism, so that taking quotients produces a commutative diagram 
			\begin{center}
				\begin{tikzcd}
					Y'\arrow[r, "\pi'"]\arrow[d] & X' \arrow[d]\\
					Y \arrow[r, "\pi"] & X.
				\end{tikzcd}
			\end{center}

			Denote $$\mu = \sum_i a_i \delta_{c_i\operatorname{ord}_{F_i}}\in \mathcal{M}^{\di, G}_Y,$$ where the $F_i$ are each prime divisors on $Y'$. The prime divisors $F_i$ then have image which we denote $\pi'(F_i) = D_i$ in $X'$; we let $e_{F_i}$ denote the ramification index along $F_i$.
			
			We next calculate the pushforward measure, which by Proposition \ref{divisorial_points_map_to_div_points} and Lemma \ref{explicit-pushforward} is given by \begin{align*}\pi_* \mu &= \sum_i a_i \delta_{\pi(c_i\operatorname{ord}_{F_i})}, \\ &= \sum_i a_i \delta_{e_{F_i}c_i\operatorname{ord}_{D_i}}.\end{align*} We can now use \cite[Proof of Proposition 5.20]{kollar-mori} to conclude that the discrepancies satisfy
			\begin{equation*}
				A_Y(F_i) = e_{F_i}A_{(X,B)}(D_i),
			\end{equation*}
			implying that the entropies satisfy 
			\begin{equation*}
				\begin{split}
					\operatorname{Ent}_{(X,B)}(\pi_*\mu)&= \sum_i a_i e_{F_i}c_iA_{(X,B)}(D_i),\\
					&= \sum_i a_i c_iA_Y(F_i), \\
					&= \operatorname{Ent}_{Y}(\mu),
				\end{split} 			\end{equation*} where we use the definition of the entropy of a general divisorial valuation given in Remark \ref{extension-of-discrepancy}.

			Corollary \ref{equivalence of differentials} proves a similar inequality for the derivatives of the energies of measures involved in the definitions of $\beta(\mu)$ and $\beta(\pi_*\mu)$, where we use the Riemann--Hurwitz formula $K_Y = \pi^*\left(K_X +\left(1-1/m\right)B \right)$, proving equality of the two beta invariants.
			
			Since any divisorial measure on $X\an$ is of the form $\pi_*\mu$ for $\mu \in \mathcal{M}^{\di, G}_Y$ by Proposition \ref{isomorphism between divisorial measures}, the other direction also follows.  \end{proof}

			\section{Interpolation and applications}
			
			We next give a general situation in which one can apply Theorem \ref{intromainthm} to construct $G$-equivariantly divisorially stable varieties, and in particular prove Theorem \ref{introinterpolation} and Corollary \ref{introcor}.
			
			\subsection{Divisorial stability in the asymptotic regime} We consider a normal projective variety $X$ endowed with an ample $\Q$-line bundle $L_X$. Recall that given an effective $\Q$-divisor $B$ on $X$, we say that $(X,B)$ is \emph{log canonical} if $K_X+B$ is $\Q$-Cartier and $A_{(X,B)}  \geq 0$ on $X^{\di}$. %The proof of the following result is similar to (and relies on) \cite[Corollary 3.12]{stablemaps}, and so we only sketch the details.
				
			\begin{theorem}
			
			There is a $k>0$ such that, for any $B\in |kL|$ such that $(X,B)$ is log canonical, $((X,B);L)$ is log divisorially stable.
			
			\end{theorem}
			
			\begin{proof}
			
			Since $(X,B)$ is log canonical, for any divisorial measure $\mu$ on $X\an$ the entropy satisfies $$\operatorname{Ent}_{(X,B)}(\mu) \geq 0.$$ Thus to prove the result, it suffices to prove that there is a $k >0$ and an $\epsilon>0$ such that for all divisorial measures $\mu$ we have $$ \nabla_{K_{X}+B}\lVert\mu \rVert_{L_X} \geq \epsilon \|\mu\|_{L_X}.$$

			We next reduce to an analogous claim regarding metrics instead of measures. Recall from Example \ref{metrics-vs-measures} that if $\phi \in \E^1$ then $$M(\phi) = \beta(\MA(\phi)).$$ Boucksom--Jonsson's proof of this statement in fact proves that $$\nabla_{K_X+B} E_{L_X}(\phi) = \nabla_{K_X+B} \|\MA(\phi)\|_{L_X}.$$ Thus by Theorem \ref{BJ-CY}, namely the non-Archimedean Calabi--Yau theorem, if we can prove that there is a $k>0$ and an $\epsilon>0$ such that for all $\phi \in \E^1$ $$\nabla_{K_X+B} E_{L_X}(\phi)  \geq \epsilon \|\phi\|_{\min},$$ the proof is concluded. But this inequality on $\H^{\NA}$ is (modulo our non-Archimedean language) proven as part of \cite[Proof of Theorem 3.7]{stablemaps}, where we use that $B \in |kL_X|$. The corresponding inequality on $\E^1$ follows by continuity of the extension of mixed Monge--Amp\`ere energies from $\H^{\NA}$ to $\E^1$.\end{proof}
			
			\begin{remark}\label{whichk} As explained in the introduction, this result is the (divisorial) ``log'' version of its (metric) ``twisted'' counterpart \cite[Theorem 3.7]{stablemaps}, which in turn is the algebro-geometric counterpart of analytic results of Hashimoto \cite{hashimoto} and Zeng \cite{zeng} regarding twisted constant scalar curvature K\"ahler metrics and Aoi--Hashimoto--Zheng regarding conical constant scalar curvature K\"ahler metrics \cite[Theorem 1.10]{AHZ} when working over $\C$. Our proof is completely algebro-geometric, and the $k$ needed is explicit: for a $\Q$-line bundle $H$, setting $$\mu(X,L_X)= \frac{-K_X.L_X^{n-1}}{L_X^n},$$ a sufficient condition is that $$\frac{n}{n+1} \mu(X,L_X)L_X - K_X + \frac{k}{2n(n+1)}L_X$$ be nef (which is certainly true for $k \gg 0$). This requirement can be sharpened using analytic techniques when $X$ is smooth and $X$ and $Y$ are defined over $\C$; see Remark \ref{whichk2}.\end{remark} 
			
 We next note the following interpolation result for divisorial stability, analogous to (for example) \cite[Lemma 2.6]{dervan_2016} which concerns log K-stability. In what follows, we assume that both $K_X$ and $B$ are $\Q$-Cartier, so that $K_X+cB$ is $\Q$-Cartier for all $c \in \Q$.
		 	
			\begin{lemma} Suppose $(X,L_X)$ is divisorially semistable and $((X,B);L_X)$ is log divisorially stable. Then $((X,cB);L_X)$ is log divisorially stable for all $0 < c \leq 1$. \end{lemma}

\begin{proof}
We denote the beta invariant of a divisorial measure $\mu$ defined with respect to  $((X,cB);L_X)$ and $(X,L_X)$ by  $\beta_{((X,cB);L_X)}$  and $\beta_{(X,L_X)}$  respectively for clarity. For a divisorial measure $\mu$, we then wish to compare $\beta_{((X,cB);L_X)}(\mu)$ and $\beta_{(X,L_X)}(\mu)$, and we begin with their respective entropy terms $\Ent_{(X,cB)}(\mu)$ and $\Ent_{X}(\mu)$. 

Write $\mu = \sum_j a_j \delta_{b_j\ord F_j}$ for a finite collection of  divisorial valuations $b_j\ord F_j$ on $X$. For a single divisorial valuation $b_j\ord F_j$ with $F_j \subset Y_j$ and $\pi_j: Y_j\to X$ the associated birational morphism, we note \begin{align*}A_{(X,cB)}(b_j\ord F_j) &= b_jA_{(X,cB)}(F_j), \\ &= b_j(\ord_{F_j}(K_{Y_j} - \pi^*(K_X + cB)+1), \\ &= b_j(\ord_{F_j}(K_{Y_j} - \pi^*K_X+1)) - b_jc\ord_{F_j}(\pi^*B),\\ &= A_X(b_j\ord_{F_j})- c\left(b_j\ord_{F_j}(\pi^*B)\right).\end{align*} By linearity we thus obtain $$\Ent_{(X,cB)}(\mu) - \Ent_{X}(\mu) = -c\left(\sum_j b_j\ord_{F_j}(\pi^*B)\right).$$ 

We next consider the dependence on $c$ of the remaining term $ \nabla_{K_{X}+cB}\lVert\mu \rVert_{L_X}$ comprising $\beta_{((X,cB);L_X)}(\mu)$, which by linearity satisfies (using that $B$ is $\Q$-Cartier) $$\nabla_{K_{X}+cB}\lVert\mu \rVert_{L_X} = \nabla_{K_{X}}\lVert\mu \rVert_{L_X} +c  \nabla_{B}\lVert\mu \rVert_{L_X}.$$ Thus if we define, for a divisorial measure $\mu = \sum_j a_j \delta_{b_j\ord F_j}$, a functional $$\tilde \beta(\mu): \M^{\di} \to \R$$ by  $$\tilde \beta(\mu) =   \nabla_{B}\lVert\mu \rVert_{L_X} - \sum_jb_j\ord_{F_j}(\pi^*B),$$ then $$\beta_{((X,cB);L_X)}(\mu) = \beta_{(X,L_X)}(\mu) + c\tilde \beta(\mu).$$ 

By hypothesis there is an $\epsilon>0$ such that for all divisorial measures $\mu$ $$\beta_{(X,L_X)}(\mu) \geq 0 \textrm{ and } \beta_{((X,B);L_X)} \geq \epsilon \|\mu\|_{L_X}.$$ Writing $$\beta_{((X,cB);L_X)}(\mu) = c\left( \beta_{(X,L_X)}(\mu) + \tilde \beta(\mu)\right) + (1-c) \beta_{(X,L_X)}(\mu),$$ we next recall that by hypothesis 
%$$\beta_{(X,L_X)}(\mu)\geq 0$$ and 
$$\beta_{(X,L_X)}(\mu) + \tilde \beta(\mu) \geq \epsilon \lVert\mu \rVert_{L_X}.$$ Thus since $0 < c \leq 1$ we obtain $$\beta_{((X,cB);L_X)}(\mu) \geq c\epsilon \lVert\mu \rVert_{L_X},$$ proving log divisorial stability. \end{proof}	

The following is then an automatic consequence of these two results.

\begin{corollary}\label{body-interp-cor} Suppose $(X,L_X)$ is divisorially semistable. Then there is a $k>0$ such that, for any $B\in |kL|$ such that $(X,B)$ is log canonical and any $m>1$, $((X,(1-1/m)B);L)$ is log divisorially stable.\end{corollary}

In particular Theorem \ref{intromainthm} produces the following, which proves Theorem \ref{introinterpolation}. We continue the notation of Corollary \ref{body-interp-cor} and let $G$ be the cyclic group of order $m$.

\begin{corollary}\label{end-cor}
Under the hypotheses of Corollary \ref{body-interp-cor}, suppose $\pi: Y \to X$ is the $m$-fold branched cover over $B$, and set $L_Y = \pi^*L_X$. Then $(Y,L_Y)$ is $G$-equivariantly divisorially stable. 
\end{corollary}
			
We refer to Koll\'ar for the construction of $m$-fold branched coverings over Cartier divisors \cite[Section 2.11]{kollar}.

 			\subsection{$G$-equivariant divisorial and uniform K-stability}
		We end by using Corollary \ref{end-cor} to produce constant scalar curvature K\"ahler metrics, for which we need $Y$ to be smooth. We need the following equivariant analogue of Boucksom--Jonsson's work relating divisorial and uniform K-stability, proved using their results.
		
		\begin{theorem}\label{equiv-BJ}
		A polarised variety $(Y,L_Y)$ is $G$-equivariantly divisorially stable if and only if it is uniformly K-stable on $\E^{1,G}$.
		\end{theorem}

\begin{proof}

Let $\mu \in \M^{1,G}$ and let $\phi \in \E^{1,G}$ be such that $$\MA(\phi) = \mu,$$ where $G$-invariance of such a $\phi$ follows from Proposition \ref{G-invarinant measures and energy comparison}. By  Proposition \ref{isomorphism between metrics of finite energy}, there is a sequence $\phi_k \in \H^{\NA,G}$ of $G$-invariant Fubini--Study metrics decreasing to $\phi$, and in particular $\MA(\phi_k)$ converges to $\MA(\phi) = \mu$ by continuity of the Monge--Amp\`ere operator. Thus we have produced a sequence of $G$-invariant divisorial measures converging to $\mu$. It follows that $G$-equivariant divisorial stability is equivalent to the existence of an $\epsilon>0$ such that for all $\mu\in \M^{1,G}$ $$\beta(\mu) \geq \epsilon \|\mu\|_{L_Y},$$ by continuity of the quantities involved.

We next show that the latter condition is equivalent to  uniform K-stability on $\E^{1,G}$. As the content of their proof that divisorial stability is equivalent to uniform K-stability on $\E^1$ (along with the non-Archimedean Calabi--Yau theorem), Boucksom--Jonsson prove that for any $\phi\in \E^1$ we have equality \begin{equation}\label{BJ-eqn} M(\phi) = \beta(\MA(\phi));\end{equation} see Example \ref{metrics-vs-measures}. But  Theorem \ref{G-invarinant measures and energy comparison} implies that the Monge--Amp\`ere operator induces a bijection between finite norm $G$\emph{-invariant} measures and finite energy $G$-\emph{invariant} metrics (modulo addition of constants), meaning that Equation \eqref{BJ-eqn} proves the equivalence of the conditions $$\beta(\mu) \geq \epsilon \|\mu\|_{L_Y} \textrm{ for all } \mu \in \M^{1,G}$$ and $$M(\phi) \geq \epsilon \|\phi\|_{\min} \textrm { for all } \phi \in \E^{1,G}$$ and hence proves the result.  \end{proof}

We now take the field $k$ over which the varieties $X$ and $Y$ are defined to be $\C$. We  use the following important result of Li \cite[Theorem 1.10]{chi-li-csck}.

\begin{theorem} Suppose $Y$ is smooth. If $(Y,L_Y)$ is uniformly K-stable on $\E^{1,G}$, then  $c_1(L)$ admits a constant scalar curvature K\"ahler metric. \end{theorem}
			
From this we obtain the following consequence of Corollary \ref{end-cor}, which proves Corollary \ref{introcor}. 

\begin{theorem}
Let $(X,L_X)$ be a divisorially semistable smooth polarised variety. There is a $k>0$ such that if we let
\begin{enumerate}[(i)]
\item  $B \in |kL_X|$ be such that $(X,B)$ is log canonical,
\item and let $\pi: Y \to X$ be the $m$-fold cover of $X$ branched over $B$, 
\item and assume $Y$ is smooth,
\end{enumerate}then $Y$ admits a constant scalar curvature K\"ahler metric in $c_1(L_Y)$, where $G$ is the associated cyclic group of degree $m$ and $L_Y = \pi^*L_X$.
\end{theorem}	
			
\begin{remark}\label{whichk2}
As mentioned in the introduction, the existence of such metrics in this situation is also consequence of work of Arezzo--Della Vedova--Shi \cite{ADS}, provided one replaces the assumption that $(X,L_X)$ be divisorially semistable with the assumption that the (Archimedean) Mabuchi functional be bounded below on the space of K\"ahler metrics in $c_1(L_X)$ (which holds, in particular, when $c_1(L_X)$ admits a cscK metric). In fact, the $k$ they obtain as sufficient is slightly more general than our work, essentially as they use analytic techniques and work of Song--Weinkove \cite{song-weinkove}, whereas our work relies on the entirely algebro-geometric \cite[Theorem 3.7]{stablemaps} which is slightly weaker. We refer to \cite[Section 6]{ADS} for many examples of applications.
\end{remark}					
			
		%			
		%			From Corollary \ref{equivalence of energies reverse} we know that $ \lVert \mu \rVert_{L_Y-tK_Y} = \lVert \pi_*\mu \rVert_{L_X - tK_{X,(1-\frac{1}{m})D}}$. Thus, we have $\beta(\mu) = \beta(\pi_*\mu)$, 
		%			i.e. $(Y,L_Y)$ is divisorially (semi-)stable.
		%			
		%(NOTE: I THINK WE NEED SOMETHING LIKE "DIVISORIAL STABILITY NEEDS TO BE CHECKED ONLY ON G-INVARIANT TEST CONFIGURATIONS SIMILAR TO DATAR SZEKELYHIDI. WE CAN DISCUSS THIS)
		%			For the other direction, let $(Y,L_Y)$ be divisorially (semi-)stable. We need to show that for any $\nu \in \mathcal{M}^{\di}_X$, $\beta(\nu)\geq 0$. Let $\nu = \sum a_i \delta_{\operatorname{ord}_{D_i}}$ for some divisors $D_i\subset X'$. Then, we have $\pi^* \nu = \sum a_i \sum_{\pi(u_j) = \operatorname{ord}_{D_i}} e_j^i \delta_{u_j}$. From Proposition \ref{isomorphism between divisorial measures} we have that $\nu = \pi_*\pi^* \nu $, and hence from the first part of the proof, we have $\beta(\pi^* \nu) = \beta(\pi_*\pi^* \nu) = \beta(\nu)$, hence $((X,D);L_X)$ is divisorially (semi-)stable.		
		
		\bibliographystyle{alpha}
		\bibliography{bibliography.bib}

	\end{document}